%% file: main.tex
\documentclass{article}
\usepackage[utf8]{inputenc}
\usepackage[all,cmtip]{xy}
\usepackage{amssymb,amsmath,amsthm,MnSymbol}
\usepackage{mathtools,mathdots}
\usepackage{mathrsfs,calligra, bbm}
\usepackage{pifont}
\usepackage[top=1.5in, bottom=1.5in, left=1in, right=1in]{geometry}
\usepackage{graphicx}
\usepackage{CJKutf8}
\usepackage{hyperref}
\usepackage{tikz-cd}
\usepackage{enumitem}
\usepackage[T1]{fontenc}
\usepackage{mathrsfs}
\DeclareMathAlphabet{\mathpzc}{OT1}{pzc}{m}{it}
\usepackage[
backend=biber,
style=alphabetic,
sorting=anyt
]{biblatex}
\usepackage{lipsum}

\newcommand\scalemath[2]{\scalebox{#1}{\mbox{\ensuremath{\displaystyle #2}}}}

\DeclareMathAlphabet{\mathpzc}{OT1}{pzc}{m}{it}

\newtheorem{Definition}{Definition}[subsection]
\newtheorem{Question}[Definition]{Question}
\newtheorem{Theorem}[Definition]{Theorem}
\newtheorem{Lemma}[Definition]{Lemma}
\newtheorem{Proposition}[Definition]{Proposition}
\newtheorem{Corollary}[Definition]{Corollary}

\newtheorem{Remark}[Definition]{Remark}

\newtheorem{Theorem*}{Theorem}[]

\DeclareMathOperator\A{\mathbf{A}}

\DeclareMathOperator\C{\mathbf{C}}

\DeclareMathOperator\F{\mathbf{F}\!}
\DeclareMathOperator\K{\mathbf{K}}
\DeclareMathOperator\Q{\mathbf{Q}}
\DeclareMathOperator\R{\mathbf{R}}
\DeclareMathOperator\T{\mathbf{T}}
\DeclareMathOperator\V{\mathbf{V}}
\DeclareMathOperator\Z{\mathbf{Z}}
\DeclareMathOperator\w{\mathbf{w}}

\DeclareMathOperator\bfu{\mathbf{u}}
\DeclareMathOperator\bfx{\textbf{\textit{x}}}

\DeclareMathOperator\bbA{\mathbb{A}}
\DeclareMathOperator\bbB{\mathbb{B}}
\DeclareMathOperator\bbG{\mathbb{G}}
\DeclareMathOperator\bbH{\mathbb{H}}

\DeclareMathOperator\bbT{\mathbb{T}}

\DeclareMathOperator\calA{\mathcal{A}}

\DeclareMathOperator\calE{\mathcal{E}}

\DeclareMathOperator\calU{\mathcal{U}}
\DeclareMathOperator\calV{\mathcal{V}}
\DeclareMathOperator\calW{\mathcal{W}}
\DeclareMathOperator\calX{\mathcal{X}}
\DeclareMathOperator\calY{\mathcal{Y}}
\DeclareMathOperator\calZ{\mathcal{Z}}

\DeclareMathOperator\scrE{\mathscr{E}}
\DeclareMathOperator\scrF{\mathscr{F}}

\DeclareMathOperator\scrH{\mathscr{H}}

\DeclareMathOperator\scrL{\mathscr{L}}

\DeclareMathOperator\scrO{\mathscr{O}}

\DeclareMathOperator\scrT{\mathscr{T}}

\DeclareMathOperator\frakL{\mathfrak{L}}

\DeclareMathOperator\frakP{\mathfrak{P}}

\DeclareMathOperator\frakX{\mathfrak{X}}
\DeclareMathOperator\frakY{\mathfrak{Y}}
\DeclareMathOperator\frakZ{\mathfrak{Z}}

\DeclareMathOperator\frakd{\mathfrak{d}}

\DeclareMathOperator\frakm{\mathfrak{m}}

\DeclareMathOperator\GL{GL}
\DeclareMathOperator\GSp{GSp}

\DeclareMathOperator\End{End}
\DeclareMathOperator\Hom{Hom}

\DeclareMathOperator\ad{ad}

\DeclareMathOperator\adj{adj}
\DeclareMathOperator\AIP{AIP}
\DeclareMathOperator\alg{alg}
\DeclareMathOperator\an{an}

\DeclareMathOperator\BS{BS}

\DeclareMathOperator\cl{cl}

\DeclareMathOperator\Cone{Cone}
\DeclareMathOperator\Cont{Cont}

\DeclareMathOperator\Cov{Cov}

\DeclareMathOperator\cts{cts}

\DeclareMathOperator\diag{diag}
\DeclareMathOperator\Dist{Dist}

\DeclareMathOperator\Fil{Fil}

\DeclareMathOperator\Fitt{Fitt}

\DeclareMathOperator\fl{fl}

\DeclareMathOperator\hst{hst}

\DeclareMathOperator\image{image}
\DeclareMathOperator\Iw{Iw}

\DeclareMathOperator\mult{mult}
\DeclareMathOperator\multideal{\mathfrak{mult}}

\DeclareMathOperator\ord{ord}
\DeclareMathOperator\one{\mathbbm{1}}
\DeclareMathOperator\oneanti{\breve{\one}}
\DeclareMathOperator\opp{opp}
\DeclareMathOperator\Par{par}

\DeclareMathOperator\sd{sd}
\DeclareMathOperator\Shi{\text{\begin{CJK}{UTF8}{min}し\end{CJK}}}

\DeclareMathOperator\Spa{Spa}

\DeclareMathOperator\Spec{Spec}

\DeclareMathOperator\supp{supp}
\DeclareMathOperator\tol{tol}

\DeclareMathOperator\trans{^{\mathtt{t}}\!}

\DeclareMathOperator\wt{wt}

\DeclareMathOperator\Sch{\textbf{\textsc{Sch}}}
\DeclareMathOperator\Sets{\textbf{\textsc{Sets}}}

\DeclareMathOperator\bfalpha{\boldsymbol{\alpha}}
\DeclareMathOperator\bfbeta{\boldsymbol{\beta}}
\DeclareMathOperator\bfgamma{\boldsymbol{\gamma}}
\DeclareMathOperator\bfdelta{\boldsymbol{\delta}}

\DeclareMathOperator\bflambda{\boldsymbol{\lambda}}

\DeclareMathOperator\bftau{\boldsymbol{\tau}}
\DeclareMathOperator\bfupsilon{\boldsymbol{\upsilon}}

\DeclareMathOperator\llbrack{\![\![\!}
\DeclareMathOperator\rrbrack{\!]\!]}

\DeclareMathOperator\bls{\boldsymbol{[}}
\DeclareMathOperator\brs{\boldsymbol{]}}
\DeclareMathOperator\bla{\boldsymbol{\langle}}
\DeclareMathOperator\bra{\boldsymbol{\rangle}}

\makeatletter
\renewcommand{\maketitle}{\bgroup\setlength{\parindent}{0pt}
\begin{flushleft}
  \LARGE{\textbf{\@title}}
  
  \vspace{4mm}
  
  \large{\textsc{\@author}}
  
  \vspace{4mm}
\end{flushleft}\egroup
}
\makeatother

\title{A pairing on the cuspidal eigenvariety for $\mathrm{GSp}_{2g}$ and the ramification locus}
\author{Ju-Feng Wu}

\date{}

\begin{document}

\maketitle

%\noindent\textit{To people suffering from the coronavirus epidemy in 2019-2020.}

{\footnotesize 
\paragraph{Abstract.} In the present paper, we first construct a pairing on the space of analytic distributions associated with $\mathrm{GSp}_{2g}$. By considering the overconvergent parabolic cohomology groups and following the work of Johansson--Newton, we construct the cuspidal eigenvariety for $\mathrm{GSp}_{2g}$. The pairing on the analytic distributions then induces a pairing on some coherent sheaves of the cuspidal eigenvariety. As an application, we follow the strategy of Bella\"{i}che to study the ramification locus of the cuspidal eigenvariety over the corresponding weight space.  
}

\tableofcontents

\input{1-Introduction}

\input{2-Overconvergent}
\input{3-Eigenvariety}
\input{4-Ramification}

\printbibliography[heading=bibintoc]

\vspace{15mm}

\begin{tabular}{l}
    Concordia University   \\
    Department of Mathematics and Statistics\\
    Montr\'{e}al, Qu\'{e}bec, Canada\\
    \textit{E-mail address: }\texttt{ju-feng.wu@mail.concordia.ca }
\end{tabular}

\end{document}

%% file: 1-Introduction.tex
\section{Introduction} \addtocontents{toc}{\protect\setcounter{tocdepth}{1}}
\subsection{Overview} After the introduction of the \textit{eigencurve} by R. Coleman and B. Mazur in \cite{Coleman_Mazur}, there were many other mathematicians who contributed to its study. The eigencurve is a rigid analytic curve which parameterises overconvergent Hecke eigenforms of finite slope and its geometry is very interesting and quite mysterious. For example we don't know even in an example if the eigencurve has finitely or infinitely many irreducible components. \\

It is a natural question to ask whether one can generalise the notion of the eigencurve to $p$-adic automorphic forms on other Shimura variety. In \cite{AIP-2015}, F. Andreatta, A. Iovita and V. Pilloni construct sheaves of (families of) overconvergent Siegel modular forms. As an application, they construct an \textit{eiegnvariety} $\calE^{\mathrm{AIP}}$, parametrising overconvergent cuspidal Siegel eigenforms of finite slope and raise the following question:

\begin{Question}[$\text{\cite[Open Problem 1]{AIP-2015}}$]\label{Question AIP}
Let $\calW$ be the weight space. Is the weight map $\calE^{\AIP}\rightarrow \calW$ unramified at classical points?
\end{Question}

On the other hand, G. Stevens introduced in \cite{Stevens-MS} the overconvergent modular symbols as a new tool to study the eigencurve, method of study which was taken over by other authors, for instance, \cite{Bellaiche-critical}, \cite[Chapter VIII]{Bellaiche-eigenbook}, \cite{Kim} and \cite{Park}. By using the overconvergent modular symbols, W. Kim constructed a pairing on the cuspidal eigencurve and applied it to the study of the ramification locus of the cuspidal eigencurve over the weight space in his Ph.D. thesis \cite{Kim}. Such a construction and results were rewritten in a more conceptual way by J. Bella\"{i}che in \cite[Chapter VIII]{Bellaiche-eigenbook}.\\

The idea of overconvergent modular symbols turns out to be a powerful tool to generalise the results to more general reductive groups by considering the overconvergent cohomology. One names \cite{Ash-Stevens}, \cite{Urban-2011} and \cite{Hansen-PhD} for the generalisation in such a direction. Furthermore, C. Johansson and J. Newton stepped further to give such a formalism in the language of adic spaces in \cite{Johansson-Newton}, which consequently allows one to read the information over the $p=0$ locus of the weight space.\\

The work presented in this paper is motivated by Question \ref{Question AIP} and is highly inspired by \cite[Chapter VIII]{Bellaiche-eigenbook} and \cite{Johansson-Newton}. More precisely, we construct a pairing on the overconvergent cohomology groups for $\GSp_{2g}$ and adapt the formalism in \textit{op. cit.} to construct the corresponding cuspidal eiegnvariety $\calE_0$ by working with the parabolic cohomology. Following the philosophy presented in \cite[Chapter VIII]{Bellaiche-eigenbook}, we attempt to use such a pairing to detect the ramification locus of $\calE_0$ over the weight space $\calW$, aiming to provide a (partial) answer to Question \ref{Question AIP}. \\

The results of the presenting paper are summarised in the following theorem. 

\begin{Theorem}\label{Theorem: summary of the paper}
Fix a $g\in \Z_{>0}$, an odd prime number $p$ and an integer $N>3$ such that $p\nmid N$. Let $X_{\Iw^+}(\C)$ be the $\C$-points of the Siegel modular variety parametrising principally polarised abelian varieties of genus $g$ with level structure given by \begin{align*}
    \Gamma(N) & =\{\bfgamma\in \GSp_{2g}(\widehat{\Z}): \bfgamma\equiv \one_g\mod N\} \quad \text{ and }\\ 
    \Iw_{\GSp_{2g}}^+ & = \left\{\bfgamma\in \GSp_{2g}(\Z_p): \bfgamma\equiv \left(\begin{array}{ccc|ccc}
            * &&  & * &\cdots & * \\ & \ddots & & \vdots & & \vdots \\ && * & * & \cdots & *\\ \hline &&& * \\ &&&& \ddots \\ &&&&& * 
        \end{array}\right)\mod p\right\}.
\end{align*} We have the following\begin{enumerate}
    \item There exists a cuspidal eigenvariety $\calE_{0}\xrightarrow{ \wt}\calW$, parametrising the eigenvectors of finite slope parabolic cohomology groups $H^{\bullet}_{\Par}(X_{\Iw^+}(\C), D_{\kappa}^{\dagger})^{\leq h}$. Denote by $\calE$ the eigenvariety constructed in \cite{Johansson-Newton} by using the algebraic group $\GSp_{2g}$, then there is a closed immersion $\calE_0\hookrightarrow \calE$ of adic spaces over $\calW$ and $\calE_0$ is the cuspidal part of $\calE$. (see \S \ref{subsection: cuspidal eigenvariety}) 
    
    \item Let $\calZ$ be the Fredholm hypersurface in \S \ref{subsection: cuspidal eigenvariety} and let $\pi^{\calE_0}_{\calZ}: \calE_0\rightarrow \calZ$ be the structure morphism. Let $\scrH_{\Par}^{\tol}$ be the coherent sheaf associated to (finite-slope) total parabolic cohomology groups $H_{\Par}^{\tol}(X_{\Iw^+}(\C), D_{\kappa}^{\dagger})^{<h}$ on $\calZ$. Then there is a pairing $$(\pi^{\calE_0}_{\calZ})^*\scrH^{\tol}_{\Par}\times (\pi^{\calE_0}_{\calZ})^*\scrH^{\tol}_{\Par}\rightarrow \scrO_{\calE_0}\quad (\text{resp. }\scrH^{\tol}_{\Par}\times \scrH^{\tol}_{\Par}\rightarrow \scrO_{\calZ})$$ of coherent sheaves on $\calE_0$ (resp., $\calZ$). (see Corollary \ref{Corollary: pairing on the eigenvariety})
    
    \item Suppose $\bfx\in \calE_0^{\fl}$ is a good classical point (see Corollary \ref{Corollary: main result of the paper}) Then there exists a function $L^{\adj}_{\calV}$ on a small enough clean neighbourhood $\calV$ of $\bfx$, determined uniquely by the above pairing up to a unit in the eigenalgebra, such that $$L^{\adj}_{\calV}(\bfx)=0\text{ if and only if }\wt \text{is ramified at }\bfx.$$ (see Theorem \ref{Theorem: ramification locus} and Corollary \ref{Corollary: main result of the paper})
    
    \item Retain the situation as above and assume further that $\bfx$ is a smooth point in $\calE_0^{\fl}$. Let $e(\bfx)$ be the quantity depending on $\bfx$ and $\wt$ defined in Theorem \ref{Theorem: vanishing order and the ramification index}, then $$\ord_{\bfx}L^{\adj}=e(\bfx).$$ (see Theorem \ref{Theorem: vanishing order and the ramification index} and Corollary \ref{Corollary: main result of the paper})
\end{enumerate}
\end{Theorem}

\subsection{Some remarks} The works presented in this paper have their connections to some known results. We summarise them in the following remarks: \begin{enumerate}
    \item We should remark first that the $p$-adic subgroup considered in this paper is slightly different from the other authors. This is due to an issue when constructing the pairing. However, the underlying distribution spaces $D_{\kappa}^r(\T_0, R)$ appearing in the present paper are isomorphic to the ones considered in \cite{Johansson-Newton} (in the case of $\GSp_{2g}$). Thus, the strategies of the known literature can go through after changing the underlying locally symmetric space. In particular, one can expect the comparisons of eigenvarieties in the next remark.
    
    \item There are comparisons of the cuspidal eigenvariety $\calE_0$ considered in this paper with the other eigenvarietyies constructed by others. As mentioned above, $\calE_0$ is the cuspidal part of the eigenvariety $\calE$ constructed in \cite{Johansson-Newton} when considering $\GSp_{2g}$. By [\textit{op. cit.}, Remark 4.1.9], the eigenvariety $\calE^{\text{Han}}$ constructed by D. Hansen in \cite{Hansen-PhD} is the open locus of $\calE$ on which $p\neq 0$. On the other hand, there is a closed immersion mapping from the eigenvariety $\calE^{\text{Urb}}$ constructed by E. Urban in \cite{Urban-2011} to $\calE^{\text{Han}}$ by the introduction of \cite{Hansen-PhD}. Denote by $\calE^{\text{Urb}}_0$ the cuspidal part of $\calE^{\text{Urb}}$ and let $\calE^{\text{Urb}}_{0, \text{red}}$ be the reduced cuspidal eigenvariety of $\calE^{\text{Urb}}$, then it coincides with the eigenvariety $\calE^{\text{AIP}}$ constructed in \cite{AIP-2015} (see \cite[pp. 627]{AIP-2015}). In conclusion, the comparisons among these eigenvarieties can be summarised in the following diagram $$\begin{tikzcd}
    \calE^{\text{AIP}}\arrow[r, equal] & \calE^{\text{Urb}}_{0, \text{red}}\arrow[rrr, hook, "\text{closed immersion}"]\arrow[d, hook, "\substack{\text{reduced}\\ \text{cuspidal part}}"'] & & & \calE_0^{\text{Han}}\arrow[rrr, hook, "\text{open immersion}", "p\neq 0"']\arrow[d, hook, "\substack{\text{cuspidal}\\ \text{part}}"'] & & & \calE_0\arrow[d, hook, "\substack{\text{cuspidal}\\ \text{part}}"]\\
    & \calE^{\text{Urb}}\arrow[rrr, hook, "\text{closed immersion}"] & & & \calE^{\text{Han}}\arrow[rrr, hook, "p\neq 0"', "\text{open immersion}"] & & & \calE
    \end{tikzcd}.$$
\end{enumerate}

During the study of the present work, we also encounter the following (natural) questions that are worth for further studies:\begin{enumerate}[resume]
    \item The function $L^{\adj}$ in the $\GL_2$ case was justified to $p$-adically interpolate the adjoint $L$-values associated to a family of eigen-newforms in \cite{Kim}. We analogously call $L^{\adj}$ an ``adjoint $p$-adic $L$-function'' in our case, hence the justification is required . In \cite[\S 12]{Genestier-Tilouine}, A. Genestier and J. Tilouine established a pairing of the parabolic cohomology groups for $\GSp_4$ by using the symplectic pairing on the algebraic representations of $\GSp_{4}$ and the cup product. Such a pairing is related to the cardinality of the Selmer group for the adjoint Galois representation attached to $\GSp_{4}$ by [\textit{op.cit.}, Th\'{e}or\`{e}me 12.0.1]. Moreover, the cardinality of the Selmer group is suggested to be related to the adjoint $L$-value in the discussion after \textit{loc. cit.}. This suggested that once one can relate our pairing with the pairing introduced in [\textit{op. cit.}, \S 12], then the justification of the name can be done. However, such a relation is unknown to us due to the fact that the authors of \cite{Genestier-Tilouine} work with the Siegel modular variety without level structure at $p$ (in order to apply the Taylor--Wiles method) while the Siegel modular variety with $\Iw_{\GSp_{2g}}^+$-level structure at $p$ is considered in our case. The comparison between the Petersson norms of a Siegel modular form of prime-to-$p$ level and of a $p$-stabililsed Siegel modular form is conjectured to be involved in solving this problem.
    
    \item A key ingredient to obtain the results in Corollary \ref{Corollary: main result of the paper} is the non-degeneracy of the pairing. We remark that the author of \cite{Bellaiche-eigenbook} can prove such a non-degeneracy for more general weights in the $\GL_2$ case while we can only show this for classical weights. It is difficult to adapt the proof in \textit{op. cit.} since it relies on a straight forward computation and such a computation becomes messier and messier as $g$ grows. 
    
    \item  Following the strategy in \cite{Bellaiche-eigenbook}, we defined the notion of ``good points''. The author of \textit{op. cit.} could show that the set of good points is not empty in the case of $\GL_2$ by using the Eichler--Shimura isomorphism. However, we do not know if the set of good points in our situation is nonempty. 
\end{enumerate}

\paragraph{Outline of the paper.} The presenting work is organised as follows. In Section 2, we introduce the coefficients of the cohomology groups in our concern, the analytic distributions. Our modules of analytic distributions are defined by combining the formalisms in both \cite{AIS-2015} and \cite{Johansson-Newton}. Then, a pairing on the analytic distributions is constructed. Such a pairing is essential for the pairing on the overconvergent cohomology groups in Section 3. Additionally, we also construct the cuspidal eigenvariety in Section 3 after recalling a sufficient amount of terminologies from \cite{Johansson-Newton}. As mentioned before, our cuspidal eigenvariety sits inside the whole eigenvariety constructed in \textit{op. cit.} by considering the algebraic group $\GSp_{2g}$. In the final section, we apply the pairing we constructed to detect the ramification locus of the cuspidal eigenvariety over the weight space by following the strategy of \cite[Chapter VIII]{Bellaiche-eigenbook} closely.

\paragraph{Acknowledgement.} The present work is part of the author's Ph.D. project. He would like to thank his advisors Adrian Iovita and Giovanni Rosso for their enormously useful discussions regarding the materials in this paper. The author is grateful to Francesc Gispert for his careful reading and suggestions on the early drafts of this paper. The author would also like to thank Lennart Gehrmann for interesting conversations during the work of this article. %Last but not least, the author thanks the anonymous referee for the constructive suggestions and comments. 

\paragraph{Notations and conventions} Throughout this paper, we fix the following notations and conventions: \begin{enumerate}
    \item[$\bullet$] $g\in \Z_{>0}$ (we are in particular interested in the case when $g>1$)
    \item[$\bullet$] $p\in \Z_{> 0}$ an odd prime number
    \item[$\bullet$] we fix once and forever an algebraic isomorphism $\C_p\simeq \C$
    \item[$\bullet$] for any matrix $\bfalpha$, we write $\trans\bfalpha$ for its transpose
    \item[$\bullet$] for any $n\in \Z_{>0}$, we denote by $\one_n$ the $n\times n$ identity matrix and by $\oneanti_n$ the $n\times n$-matrix whose entries are $1$ on the anti-diagonal positions and $0$ elsewhere, \textit{i.e.}, $$\oneanti_n=\begin{pmatrix} & & 1\\ & \iddots & \\ 1 & &\end{pmatrix}$$
    \item[$\bullet$] in principle, symbols in Gothic font (\textit{e.g.}, $\frakX, \frakY, \frakZ$) stand for formal schemes; symbols in calligraphic font (\textit{e.g.}, $\calX, \calY, \calZ$) stand for adic spaces; and symbols in script font (\textit{e.g.}, $\scrO, \scrF, \scrE$) stand for sheaves (over certain geometric object).
\end{enumerate}

\addtocontents{toc}{\protect\setcounter{tocdepth}{2}}

%% file: 2-Overconvergent.tex
\section{Analytic distributions}\label{section: distributions}
\subsection{Algebraic and \texorpdfstring{$p$}{p}-adic groups}
Let $\V_{\Z}$ be the finite free $\Z$-module $\Z^{2g}$. By viewing elements in $\V_{\Z}$ as column vectors, we equip $\V_{\Z}$ with the symplectic pairing
\[
\bla\cdot, \cdot \bra: \V_{\Z}\times \V_{\Z}\rightarrow \Z, \quad (\vec{v}, \vec{v'})\mapsto \trans\vec{v}\begin{pmatrix} & -\oneanti_g\\ \oneanti_g\end{pmatrix}\vec{v}'.
\]
Then, the algebraic group $\GSp_{2g}$ is defined to be the subgroup of the automorphisms on $\V_{\Z}$ that preserves this pairing up to a unit. In particular, for any ring $R$,
\[\scalemath{1}{
\GSp_{2g}(R) : = \left\{\bfgamma\in \GL_{2g}(R): \trans\bfgamma\begin{pmatrix} & -\oneanti_g\\ \oneanti_g\end{pmatrix}\bfgamma = \varsigma(\bfgamma)\begin{pmatrix} & -\oneanti_g\\ \oneanti_g\end{pmatrix}\text{ for some }\varsigma(\bfgamma)\in R^\times\right\}.}
\]
Equivalently, for any $\bfgamma=\begin{pmatrix}\bfgamma_a & \bfgamma_b\\ \bfgamma_c & \bfgamma_d\end{pmatrix}\in \GL_{2g}$, $\bfgamma\in \GSp_{2g}$ if and only if  \begin{align*}
    & \trans\bfgamma_a\oneanti_g\bfgamma_c=\trans\bfgamma_c\oneanti_g\bfgamma_a,\\
    & \trans\bfgamma_b\oneanti_g\bfgamma_d=\trans\bfgamma_d\oneanti_g\bfgamma_b,\text{ and }\\
    & \trans\bfgamma_a\oneanti_g\bfgamma_d-\trans\bfgamma_c\oneanti_g\bfgamma_b=\varsigma(\bfgamma)\oneanti_g \text{ for some $\varsigma(\bfgamma)\in \bbG_m$.}
\end{align*} \\

In the present paper, we shall be considering the following algebraic and $p$-adic subgroups of $\GL_g$ and $\GSp_{2g}$: \begin{enumerate}
    \item[$\bullet$] We consider the Borel subgroups $B_{\GL_g}$ and $B_{\GSp_{2g}}$ for $\GL_g$ and $\GSp_{2g}$ respectively, defined by \begin{align*}
        B_{\GL_g} & := \text{the Borel subgroup of upper triangular matrices in $\GL_g$}\\
        B_{\GSp_{2g}} & := \text{the Borel subgroup of upper triangular matrices in $\GSp_{2g}$.}
    \end{align*} Remark that one can take the Borel subgroup of upper triangular matrices in $\GSp_{2g}$ because of the choice of the symplectic pairing on $\V_{\Z}$.
    
    \item[$\bullet$] The corresponding unipotent radicals are of the form \begin{align*}
        U_{\GL_g} := & \text{ the upper triangular $g\times g$ matrices}\\
        & \text{ whose diagonal entries are all $1$}\\
        U_{\GSp_{2g}} := & \text{ the upper triangular $2g\times 2g$ matrices in $\GSp_{2g}$}\\
        & \text{ whose diagonal entries are all $1$}.
    \end{align*} 
    
    \item[$\bullet$] The maximal tori for both algebraic groups are considered to be the maximal algebraic tori of diagonal matrices. Then the Levi decomposition yields $$B_{\GL_g}=U_{\GL_g}T_{\GL_g}\text{ and }B_{\GSp_{2g}}=U_{\GSp_{2g}}T_{\GSp_{2g}}.$$
    
    \item[$\bullet$] Denote by $U_{\GL_g}^{\opp}$ and $U_{\GSp_{2g}}^{\opp}$ the opposite unipotent radical of $U_{\GL_g}$ and $U_{\GSp_{2g}}$ respectively.
    \item[$\bullet$] To simplify the notation, we write \[
        \begin{array}{cc}
            T_{\GL_g, 0}=T_{\GL_g}(\Z_p),&  U_{\GL_g, 0}=U_{\GL_g}(\Z_p),\\ T_{\GSp_{2g}, 0}=T_{\GSp_{2g}}(\Z_p) & U_{\GSp_{2g}, 0}=U_{\GSp_{2g}}(\Z_p). 
        \end{array}
    \]For any $s\in \Z_{>0}$, we define \begin{align*}
        T_{\GL_g, s}& :=\ker(T_{\GL_g}(\Z_p)\rightarrow T_{\GL_g}(\Z/p^s\Z)), \\
        U_{\GL_g, s} & :=\ker(U_{\GL_g}(\Z_p)\rightarrow U_{\GL_g}(\Z/p^s\Z))\\
        T_{\GSp_{2g}, s}& :=\ker(T_{\GSp_{2g}}(\Z_p)\rightarrow T_{\GSp_{2g}}(\Z/p^s\Z)), \\
        U_{\GSp_{2g}, s} & :=\ker(U_{\GSp_{2g}}(\Z_p)\rightarrow U_{\GSp_{2g}}(\Z/p^s\Z)),
    \end{align*} where the all maps above are reduction maps. 
    
    \item[$\bullet$] The Iwahori subgroups of $\GL_g(\Z_p)$ and $\GSp_{2g}(\Z_p)$ are \begin{align*}
        \Iw_{\GL_g} := & \text{ the preimage of $B_{\GL_g}(\F_p)$ under the reduction map}\\
        & \text{ $\GL_g(\Z_p)\rightarrow \GL_g(\F_p)$}\\
        \Iw_{\GSp_{2g}} := & \text{ the preimage of $B_{\GSp_{2g}}(\F_p)$ under the reduction map}\\
        & \text{ $\GSp_{2g}(\Z_p)\rightarrow \GSp_{2g}(\F_p)$}.
    \end{align*} Then the Iwahori decomposition gives $$\Iw_{\GL_g}=U_{\GL_g, 1}^{\opp} T_{\GL_g, 0}U_{\GL_g, 0}\text{ and }\Iw_{\GSp_{2g}} = U_{\GSp_{2g}, 1}^{\opp} T_{\GSp_{2g}, 0}U_{\GSp_{2g}, 0},$$ where $U_{\GL_g, s}^{\opp}$ and $U_{\GSp_{2g}, s}^{\opp}$ are defined in the same way as above for any $s\in \Z_{>0}$.
    
    \item[$\bullet$] We shall consider the ``\textit{strict} Iwahori subgroups'' of $\GL_g(\Z_p)$ and $\GSp_{2g}(\Z_p)$, defined as \begin{align*}
        \Iw_{\GL_g}^+  : = & \text{ the preimage of $T_{\GL_g}(\F_p)$ under the reduction map}\\
        & \text{ $\GL_g(\Z_p)\rightarrow \GL_g(\F_p)$}\\
        \Iw_{\GSp_{2g}}^+ := & \left\{\bfgamma\in \GSp_{2g}(\Z_p): \bfgamma\equiv \left(\begin{array}{ccc|ccc}
            * &&  & * &\cdots & * \\ & \ddots & & \vdots & & \vdots \\ && * & * & \cdots & *\\ \hline &&& * \\ &&&& \ddots \\ &&&&& * 
        \end{array}\right)\mod p\right\}
    \end{align*} We caution the readers that the strict Iwahori subgroups $\Iw_{\GL_g}^+$ and $\Iw_{\GSp_{2g}}^+$ are not defined analogously. We abuse the similar symbol to simplify the notations. \\
    
    Observe that for any $\begin{pmatrix}\bfgamma_a & \bfgamma_b\\ \bfgamma_c & \bfgamma_d\end{pmatrix}\in \Iw_{\GSp_{2g}}^+$, we have $\bfgamma_a\in \Iw_{\GL_g}$. Moreover, $\Iw_{\GL_g}^+$ is stable under transpose. \\
    
    Obviously, we have $\Iw_{\GL_g}^+\subset \Iw_{\GL_g}$ and $\Iw_{\GSp_{2g}}^+\subset \Iw_{\GSp_{2g}}$. Thus, the Iwahori decompositions for $\Iw_{\GL_g}$ and $\Iw_{\GSp_{2g}}$ induce the Iwahori decompositions for $\Iw_{\GL_g}^+$ and $\Iw_{\GSp_{2g}}^+$ :
    \begin{align*}
        \Iw_{\GL_g}^+ & = U_{\GL_g, 1}^{\opp} T_{\GL_g, 0}U_{\GL_g, 1}\\
        \Iw_{\GSp_{2g}}^+ & = U_{\GSp_{2g}, 1}^{\opp} T_{\GSp_{2g}, 0} U_{\GSp_{2g}, 0}^+,
    \end{align*} where $U_{\GSp_{2g}, 0}^+ = \Iw_{\GSp_{2g}}^+ \cap U_{\GSp_{2g}, 0}$.
\end{enumerate}

\subsection{Analytic distributions}
The analytic functions and distributions that will be considered here are heavily inspired by the notions in \cite{AIS-2015} and \cite{Johansson-Newton}; we indeed combine their ideas to define the objects that we are interested in. Before everything, we recall the terminology of \textit{Banach--Tate $\Z_p$-algebra} defined in \cite{Johansson-Newton}:

\begin{Definition}\label{Definition: Banach-Tate Z_p-algebra}
A $\Z_p$-algebra $R$ (with the structure morphism $\Z_p\rightarrow R$) is a \textbf{Banach--Tate $\Z_p$-algebra} if and only if it satisfies the following properties \begin{enumerate}
    \item $R$ is a completed normed ring with norm $|\cdot|_R$; 
    \item there exists a multiplicative pseudouniformiser $\varpi\in R$, i.e., $\varpi\in R^{\times}$ such that $|\varpi|_R<1$ and $|\varpi a|=|\varpi|_R|a|_R$ for any $a\in R$; and
    \item the structure morphism $\Z_p\rightarrow R$ is norm-decreasing where we equip $\Z_p$ with the usual norm $|x|=p^{-v_p(x)}$.
\end{enumerate}
\end{Definition}

Let $R$ be a complete Tate $\Z_p$-algebra, \textit{i.e.}, $R$ is a complete topological ring admitting a multiplicative pseudouniformiser $\varpi$. We assume that $R$ admits a noetherian ring of definition. Notice that $U_{\GSp_{2g}, 1}^{\opp}\simeq \Z_p^{d_0}$ as a $p$-adic manifold for some $d_0\in \Z_{>0}$, then we consider the continuous functions and distributions \begin{align*}
    & \Cont(U_{\GSp_{2g}, 1}^{\opp}, R):=\{f: U_{\GSp_{2g}, 1}^{\opp}\rightarrow R : f\text{ is continuous}\}\\
    & \Dist(U_{\GSp_{2g}, 1}^{\opp}, R):=\Hom_R^{\cts}(\Cont(U_{\GSp_{2g}, 1}^{\opp}, R), R).
\end{align*}\\

On the other hand, define $$\T_0:=\left\{(\bfgamma, \bfupsilon)\in \Iw_{\GL_g}^+\times M_g(p\Z_p): \trans\bfgamma\oneanti_g\bfupsilon= \trans\bfupsilon\oneanti_g\bfgamma\right\}.$$ The defining condition of $\T_0$ means that there exists $\bfalpha_b, \bfalpha_d\in M_g(\Z_p)$ such that $$\begin{pmatrix}\bfgamma & \bfalpha_b\\ \bfupsilon & \bfalpha_d\end{pmatrix}\in \GSp_{2g}(\Q_p)\cap M_{2g}(\Z_p).$$ In fact, $(\bfgamma, \bfupsilon)\in \T_0$ can be seen as an element in $\Iw_{\GSp_{2g}}^+$ via $$\T_0\ni (\bfgamma, \bfupsilon)\mapsto \begin{pmatrix}\bfgamma & \\ \bfupsilon & \oneanti_g \trans\bfgamma^{-1}\oneanti_g\end{pmatrix}\in \Iw_{\GSp_{2g}}^+.$$ Moreover, one can view $\T_0$ as a $p$-adic closed submanifold  of $\Iw_{\GL_g}^+\times M_g(p\Z_p)$. \\

There are two actions on $\T_0$: \begin{enumerate}
    \item The right action of $B_{\GL_g,0}^+ := T_{\GL_g, 0}U_{\GL_g, 1}$ on $\T_0$ is defined by 
    \[
    \T_0\times B_{\GL_g, 0}^+ \rightarrow \T_0, \quad ((\bfgamma, \bfupsilon), \bfbeta)\mapsto (\bfgamma\bfbeta, \bfupsilon\bfbeta).
    \] Indeed, by embedding $B_{\GL_g, 0}^+$ into $\Iw_{\GSp_{2g}}^+$ via $\bfbeta\mapsto \begin{pmatrix}\bfbeta & \\ & \oneanti_g \trans\bfbeta^{-1}\oneanti_g\end{pmatrix}$, this action is given by 
    \[
    \begin{pmatrix} \bfgamma & *\\ \bfupsilon & *\end{pmatrix}\begin{pmatrix}\bfbeta & \\ &  \oneanti_g \trans\bfbeta^{-1}\oneanti_g\end{pmatrix} = \begin{pmatrix} \bfgamma\bfbeta & * \\ \bfupsilon\bfbeta & *\end{pmatrix}.
    \]
    
    \item The left action of $\Xi : = \begin{pmatrix} \Iw_{\GL_g}^+ & M_g(\Z_p)\\ M_g(p\Z_p) & M_g(\Z_p)\end{pmatrix}\cap \GSp_{2g}(\Q_p)$ on $\T_0$ is defined by 
    \[
    \Xi \times \T_0 \rightarrow \T_0, \quad \left(\begin{pmatrix}\bfalpha_a & \bfalpha_b\\ \bfalpha_c & \bfalpha_d\end{pmatrix}, (\bfgamma, \bfupsilon)\right) \mapsto (\bfalpha_a\bfgamma + \bfalpha_b\bfupsilon, \bfalpha_c\bfgamma + \bfalpha_d\bfupsilon).
    \] To verify this is indeed a left action, one considers 
    \[
    \begin{pmatrix}\bfalpha_a & \bfalpha_b\\ \bfalpha_c & \bfalpha_d\end{pmatrix}\begin{pmatrix}\bfgamma & *\\ \bfupsilon & *\end{pmatrix} = \begin{pmatrix} \bfalpha_a\bfgamma + \bfalpha_b\bfupsilon & * \\ \bfalpha_c\bfgamma + \bfalpha_d\bfupsilon  & *\end{pmatrix}.
    \] In particular, $\T_0$ admits a left action of $\Iw_{\GSp_{2g}}^+$ as $\Iw_{\GSp_{2g}}^+\subset \Xi$.
\end{enumerate}

\vspace{3mm}

Inside $\T_0$, there is a special subset $$\T_{00} : = \{(\bfgamma, \bfupsilon)\in \T_0: \bfgamma \in U_{\GL_g, 1}^{\opp}\}.$$ One can identify $\T_{00}$ with $U_{\GSp_{2g}, 1}^{\opp}$ via \[
\T_{00} \xrightarrow{\sim} U_{\GSp_{2g}, 1}^{\opp}, \quad (\bfgamma, \bfupsilon)\mapsto \begin{pmatrix} \bfgamma \\ \bfupsilon & \oneanti_g \trans\bfgamma^{-1}\oneanti_g\end{pmatrix}.
\]\\

Let $\kappa: T_{\GL_g, 0}\rightarrow R^\times$ be a $p$-adic weight and we assume that one can choose a norm $|\cdot|_R$ on $R$ making $R$ a Banach--Tate $\Z_p$-algebra and that $|\cdot|_R$ is adapted to $\kappa$, \textit{i.e.}, the norm $|\cdot|_R$ satisfies \begin{enumerate}
    \item[$\bullet$] $\kappa(T_{\GL_g, 0})\subset R_0:=$ the unit ball of $R$ with respect to $|\cdot|_R$ and
    \item[$\bullet$] $|\kappa(\bftau)-1|_R<1$ for all $\bftau\in T_{\GL_g, 1}$. 
\end{enumerate} We write $r_{\kappa}:=\min\{r\in [1/p, 1): |\kappa(\bftau)-1|_R\leq r\text{ for all }\bftau\in T_{\GL_g, 1}\}$. Finally, define \[
\scalemath{1}{A_{\kappa}(\T_0, R):=\left\{f: \T_0\rightarrow R: \begin{array}{l}
    f \text{ is continuous }  \\
    f(\bfgamma\bfbeta, \bfupsilon\bfbeta)=\kappa(\bfbeta)f(\bfgamma, \bfupsilon)\quad \forall (\bfgamma, \bfupsilon)\in \T_0,\,\, \bfbeta\in B_{\GL_g, 0}^+
\end{array}\right\}.}
\]One sees immediately that there is an isomorphism $$A_{\kappa}(\T_0, R)\xrightarrow{ \sim} \Cont(U_{\GSp_{2g}, 1}^{\opp}, R), \quad f\mapsto f|_{\T_{00}}.$$ 

\begin{Remark}\label{Remark: comparison with J-N}
\normalfont Our continuous functions $A_{\kappa}(\T_0, R)$ is the same as the continuous functions ``$\calA_{\kappa}$'' defined in \cite{Johansson-Newton} in the case of $\GSp_{2g}$, for which we recall $$\calA_{\kappa}:=\left\{f: \Iw_{\GSp_{2g}}\rightarrow R: \begin{array}{l}
    f\text{ is continuous}   \\
    f(\bfgamma\bfbeta)=\kappa(\bfbeta)f(\bfgamma)\,\,\forall (\bfgamma, \bfbeta)\in \Iw_{\GSp_{2g}}\times B_{\GSp_{2g}, 0}  
\end{array}\right\}.$$ By the restriction to $U_{\GSp_{2g}, 1}^{\opp}$, we have an isomorphism $\calA_{\kappa}\simeq \Cont(U_{\GSp_{2g}, 1}^{\opp}, R)$ and hence $\calA_{\kappa}$ is isomorphic to $A_{\kappa}(\T_0, R)$. We chose to work in this way due to some technicality when defining the pairing in \S \ref{subsection: pairing on the analytic distributions}. We remark that the continuous functions $\calA_{\kappa}$ considered in \textit{op. cit.} has its advantage for considering analytic functions and distributions for general reductive groups. We should also point out that we are not considering general weights associated to $T_{\GSp_{2g}, 0}$ but weights associated to $T_{\GL_g, 0}$ via the embedding \[T_{\GL_g, 0}\hookrightarrow \begin{pmatrix}T_{\GL_g, 0} & \\ & \oneanti_g T_{\GL_g, 0}^{-1}\oneanti_g\end{pmatrix}\subset T_{\GSp_{2g}, 0}.\] This explains why we can use $\T_0$ to rewrite the continuous functions $\calA_{\kappa}$ considered in \textit{op. cit.}.
\end{Remark}

\vspace{3mm}

Evidently, we define $$D_{\kappa}(\T_0, R):=\Hom_R^{\cts}(A_{\kappa}(\T_0, R), R).$$ Then we have a sequence of isomorphisms $$R\llbrack U_{\GSp_{2g}, 1}^{\opp}\rrbrack\xrightarrow{ \sim }\Dist(U_{\GSp_{2g},1}^{\opp}, R)\xrightarrow{ \sim }D_{\kappa}(\T_0, R),$$ where the last isomorphism obviously follows from the isomorphism $A_{\kappa}(\T_0, R)\simeq \Cont(U_{\GSp_{2g}, 1}^{\opp}, R)$. The first isomorphism follows from \cite[Proposition 3.1.4]{Johansson-Newton} for which one sends each $\bfgamma\in U_{\GSp_{2g}, 1}^{\opp}$ to $$\delta_{\bfgamma}:=\text{ the Dirac distribution at $\bfgamma$},$$ \textit{i.e.}, the evaluation at $\bfgamma$. The ring structure on $\Dist(U_{\GSp_{2g}, 1}^{\opp}, R)$ is given by the usual convolution product; that is, $$\mu_1*\mu_2: f\mapsto \int_{\bfgamma_1\in U_{\GSp_{2g}, 1}^{\opp}}\int_{\bfgamma_2\in U_{\GSp_{2g}, 1}^{\opp}} f(\bfgamma_1\bfgamma_2)\quad \mu_2(\bfgamma_2)\mu_1(\bfgamma_1),$$ which yields $\delta_{\bfgamma_1}*\delta_{\bfgamma_2}=\delta_{\bfgamma_1\bfgamma_2}$.\\

Recall that $U_{\GSp_{2g}, 1}^{\opp}\simeq \Z_p^{d_0}$ as $p$-adic manifolds, thus we can fix topological generators $\bfupsilon^{\circ}_1,..., \bfupsilon^{\circ}_{d_0}$ for $U_{\GSp_{2g}, 1}^{\opp}$. For $i=(i_1, ..., i_{d_0})\in \Z_{\geq 0}^{d_0}$, we write $\underline{\bfupsilon}_{\circ}^i:=(\bfupsilon^{\circ}_1-1)^{i_1}\cdots (\bfupsilon^{\circ}_{d_0}-1)^{i_{d_0}}$. Let $r\in [r_{\kappa}, 1)$, we define the $r$-norm on $R\llbrack U_{\GSp_{2g}, 1}^{\opp}\rrbrack$ by $$\left|\left|\sum_{i\in \Z_{\geq 0}^{d_0}}a_i\underline{\bfupsilon}_{\circ}^i\right|\right|_r:=\sup\left\{|a_i|_R\cdot r^{\sum_{j=1}^{d_0}i_j}: i=(i_1, ..., i_{d_0})\in \Z_{\geq 0}^{d_0}\right\}.$$ Via the above isomorphisms, we defined an $r$-norm on $D_{\kappa}(\T_0, R)$. Following [\textit{op. cit.}, \S 3], we define \begin{align*}
    & D_{\kappa}^r(\T_0, R):=\text{ the completion of $D_{\kappa}(\T_0, R)$ with respect to the $r$-norm}\\
    & D_{\kappa}^\dagger(\T_0, R):=\varprojlim_{r} D_{\kappa}^r(\T_0, R).
\end{align*}

\begin{Remark}\label{Remark: for r<s, we have a compact inclusion of distribution spaces}
\normalfont By \cite[Lemma 3.2.3]{Johansson-Newton}, for any $r<s$, one has a compact inclusion $$D_{\kappa}^s(\T_0, R)\hookrightarrow D_{\kappa}^r(\T_0, R).$$ Hence one thinks of $D_{\kappa}^{\dagger}(\T_0, R)=\cap_{r}D_{\kappa}^r(\T_0, R)$.
\end{Remark}

\begin{Remark}\label{Remark: dual of the distribution space}
\normalfont Following \cite{Johansson-Newton}, we denote by $D_{\kappa}^{r, \circ}(\T_0, R)$ the unit ball of $D_{\kappa}^{r}(\T_0, R)$. Moreover, we also consider their dual spaces $$A_{\kappa}^{<r, \circ}(\T_0, R)\quad \text{ and }\quad A_{\kappa}^{<r}(\T_0, R),$$ which can be viewed as subspaces in $A_{\kappa}(\T_0, R)$. We refer the readers to the end of [\textit{op. cit.}, \S 3] for more detail discussions. 
\end{Remark}

\subsection{A pairing on the analytic distributions}\label{subsection: pairing on the analytic distributions}
In this subsection, we establish a pairing on the analytic distributions $D_{\kappa}^r(\T_0, R)$. Our strategy is the same as the strategy as in \cite[Chapter VIII]{Bellaiche-eigenbook}. That is, we first build a map from $D_{\kappa}^r(\T_0, R)$ to $A_{\kappa}^{<r}(\T_0, R)$ and then use the natural pairing between $D_{\kappa}^r(\T_0, R)$ and $A_{\kappa}^{<r}(\T_0, R)$ to obtain the desired one.

\paragraph{An algebraic model.} Our pairing is modelled on an algebraic version inspired by the one in \cite[pp. 18]{Hansen-PhD}, which we now explain. \\

Let $k=(k_1, ..., k_g)\in \Z_{>0}^g$ with $k_1\geq \cdots \geq k_g$. One can view $k$ as a character on $T_{\GL_g}$ via $$k: T_{\GSp_{2g}}\rightarrow \bbG_m, \quad \diag(\bftau_1, ..., \bftau_g, \bftau_0\bftau_g^{-1}, ..., \bftau_0\bftau_1^{-1})\mapsto \prod_{i=1}^g\bftau_i^{k_i}.$$ One extends $k$ to $B_{\GSp_{2g}}$ by setting $k(U_{\GSp_{2g}})=\{1\}$. Consider the irreducible representation for $\GSp_{2g}$
\[ \scalemath{1}{
\V_{\GSp_{2g}, k}^{\alg} :=\left\{\phi: \GSp_{2g}\rightarrow \bbA^1: \begin{array}{l}
    \phi\text{ is a morphism of schemes}  \\
    \phi(\bfgamma\bfbeta) = k(\bfbeta)\phi(\bfgamma) \text{ for any }(\bfgamma, \bfbeta)\in \GSp_{2g}\times B_{\GSp_{2g}} 
\end{array} \right\}.}
\] One can consider the following actions of $\GSp_{2g}$ on $\V_{\GSp_{2g}, k}^{\alg}$: \begin{enumerate}
    \item[(i)] The right action given by $$(\phi\cdot \bfgamma)(\bfgamma') = \phi(\bfgamma\bfgamma').$$
    \item[(ii)]  The left action given by $$(\bfgamma\cdot \phi)(\bfgamma') = \phi(\trans\bfgamma\bfgamma').$$
    \item[(iii)] The left action given by $$(\bfgamma\cdot\phi)(\bfgamma') = \phi(\bfgamma^{-1}\bfgamma').$$
\end{enumerate} Notice that the second action is valid since $\GSp_{2g}$ is stable under transpose. In fact, one deduces easily from the definition that \[\trans\bfgamma = \varsigma(\bfgamma) \begin{pmatrix} & -\oneanti_g\\ \oneanti_g\end{pmatrix}\bfgamma^{-1}\begin{pmatrix} & \oneanti_g\\ -\oneanti_g\end{pmatrix}\] for any $\bfgamma\in \GSp_{2g}$. Therefore, the second action is nothing but a twisted action of the third one. In what follows, we equip $\V_{\GSp_{2g}, k}^{\alg}$ with the left $\GSp_{2g}$-action given by (ii).\\

Let $\V_{\GSp_{2g}, k}^{\alg, \vee}$ be its linear dual. We equip $\V_{\GSp_{2g}, k}^{\alg, \vee}$ with a left action induced by (i). Then, we have a morphism 
\[
\Phi_{k}^{\alg}: \V_{\GSp_{2g}, k}^{\alg, \vee}\rightarrow  \V_{\GSp_{2g}, k}^{\alg}, \quad \mu\mapsto \left(\bfgamma'\mapsto \int_{\bfgamma\in \GSp_{2g}}e_{k}^{\hst}(\trans\bfgamma'\bfgamma)\quad d\mu\right),
\]
where $e_k^{\hst}\in \V_{\GSp_{2g}, k}^{\alg}$ is defined by $$e_k^{\hst} : (X_{ij})_{1\leq i,j\leq 2g}\mapsto X_{11}^{k_1-k_2} \times \det\begin{pmatrix}X_{11} & X_{12}\\ X_{21} & X_{22}\end{pmatrix}^{k_2-k_3} \times \cdots \times \det\begin{pmatrix}X_{11} & \cdots & X_{1g}\\ \vdots & & \vdots \\ X_{g1} & \cdots & X_{gg}\end{pmatrix}^{k_g}.$$ One sees that $\Phi_k^{\alg}$ is $\GSp_{2g}$-equivariant with respect to the left $\GSp_{2g}$-actions on both spaces. Indeed, for any $\bfalpha, \bfgamma'\in \GSp_{2g}$ and $\mu\in \V_{\GSp_{2g}, k}^{\alg, \vee}$, we have \begin{align*}
    \Phi_k^{\alg}(\bfalpha\cdot \mu)(\bfgamma') & = \int_{\bfgamma\in \GSp_{2g}} e_k^{\hst}(\trans\bfgamma' \bfalpha \bfgamma)\quad d\mu \\
    & = \int_{\bfgamma\in \GSp_{2g}} e_{k}^{\hst}(\trans(\trans\bfalpha\bfgamma')\bfgamma)\quad d\mu \\
    & = \left(\bfalpha\cdot \Phi_{k}^{\alg}(\mu)\right)(\bfgamma').
\end{align*} Consequently, $\Phi_{k}^{\alg}$ defines a pairing on $\V_{\GSp_{2g}, k}^{\alg, \vee}$ by
\[
(\mu_1,\mu_2)\mapsto \int_{\bfgamma_1, \bfgamma_2\in \GSp_{2g}}e_k^{\hst}\left(\trans\bfgamma_2\bfgamma_1\right)\quad d\mu_1(\bfgamma_1)d\mu_2(\bfgamma_2).
\]

\begin{Remark}\label{Remark: symplecti pairing}
\normalfont Notice that $\V_{\GSp_{2g}, k}^{\alg, \vee}$ is an irreducible representation of $\GSp_{2g}$, thus it admits a pairing induced by the symplectic pairing $\bla\cdot, \cdot\bra$ on $\V_{\Z}$. This pairing can be viewed by the following formula \begin{align*}
    \bla\cdot, \cdot\bra_{k}: (\mu_1, \mu_2)\mapsto \int_{\bfgamma_1, \bfgamma_2\in \GSp_{2g}}e_{k}^{\hst}\left(\trans\bfgamma_2\begin{pmatrix} & -\oneanti_g\\ \oneanti_g\end{pmatrix}\bfgamma_1\right)\quad d\mu_1(\bfgamma_1)d\mu_2(\bfgamma_2).
\end{align*} Indeed, for any $\bfalpha\in \GSp_{2g}$, we have \begin{align*}
    \bla\bfalpha\cdot \mu_1, \mu_2\bra_k & = \scalemath{1}{\int_{\bfgamma_1, \bfgamma_2\in \GSp_{2g}}e_k^{\hst}\left(\trans\bfgamma_2 \begin{pmatrix} & -\oneanti_g\\ \oneanti_g\end{pmatrix} \bfalpha \bfgamma_1\right) \quad d\mu_1(\bfgamma_1)d\mu_2(\bfgamma_2)}\\
    & = \scalemath{1}{\int_{\bfgamma_1, \bfgamma_2\in \GSp_{2g}}e_k^{\hst}\left(\trans\bfgamma_2 \varsigma(\bfalpha)\trans\bfalpha^{-1}\begin{pmatrix} & -\oneanti_g\\ \oneanti_g\end{pmatrix} \bfgamma_1\right) \quad d\mu_1(\bfgamma_1)d\mu_2(\bfgamma_2)}\\
    & = \scalemath{1}{\varsigma(\bfalpha)^{\sum k_i}\int_{\bfgamma_1, \bfgamma_2}e_{k}^{\hst}\left(\trans(\bfalpha^{-1}\bfgamma_2)\begin{pmatrix} & -\oneanti_g\\ \oneanti_g\end{pmatrix}\bfgamma_1\right) \quad d\mu_1(\bfgamma_1)d\mu_2(\bfgamma_2)}\\
    & = \varsigma(\bfalpha)^{\sum k_i}\bla \mu_1, \bfalpha^{-1}\cdot\mu_2\bra_k,
\end{align*}
 where the second equality follows from the definition of $\GSp_{2g}$.
\end{Remark}

\paragraph{The pairing on the analytic distributions.}  Let $\kappa: T_{\GL_g, 0}\rightarrow R^\times$ be a $p$-adic weight and we keep the assumption on the fixed norm on the Banach--Tate $\Z_p$-algebra $R$ as before. Notice that we can write $$\kappa: T_{\GL_g, 0}\rightarrow R^{\times}, \quad \diag(\bftau_1, ..., \bftau_g)\mapsto \kappa_1(\bftau_1) \times \cdots \times \kappa_g(\bftau_g)$$ for some $p$-adic weight $\kappa_i:\Z_p^\times\rightarrow R^\times$. \\

Define the function $e_{\kappa}^{\hst}$ on $\Iw_{\GL_g}^+$ by 
\[
e_{\kappa}^{\hst}: (X_{ij})_{1\leq i,j\leq g}\mapsto \frac{\kappa_1(X_{11})}{\kappa_{2}(X_{11})}\times \frac{\kappa_2(\det(X_{ij})_{1\leq i,j\leq 2})}{\kappa_3(\det(X_{ij})_{1\leq i,j\leq 2})}\times \cdots \times \kappa_g(\det(X_{ij})_{1\leq i,j\leq g}).
\] One sees that $e_{\kappa}^{\hst}\in A_{\kappa}(\T_0, R)$ via $$e_{\kappa}^{\hst}(\bfgamma, \bfupsilon) = e_{\kappa}^{\hst}(\bfgamma).$$

\begin{Lemma}\label{Lemma: highest weight vector converges in the right rigion}
For $r\in [r_{\kappa}, 1)$, $e_{\kappa}^{\hst}\in A_{\kappa}^{<r}(\T_0, R)$ 
\end{Lemma}
\begin{proof}
Note that $e_{\kappa}^{\hst}(\bfgamma, \bfupsilon) = 1$ for any $(\bfgamma, \bfupsilon)\in \T_{00}$. The assertion then follows from the explicit description of $A_{\kappa}^{<r}(\T_0, R)$ in \cite[\S 3.2 \& 3.3]{Johansson-Newton}.
\end{proof}

Therefore, we can define
\begin{align*}
    \Phi_{\kappa}: D_{\kappa}^{r}(\T_0, R) & \rightarrow A_{\kappa}^{<r}(\T_0, R), \\
    \mu & \mapsto \scalemath{1}{\left((\bfgamma', \bfupsilon') \mapsto \int_{(\bfgamma, \bfupsilon)\in \T_{00}} e_{\kappa}^{\hst}\left(\begin{pmatrix}\trans\bfgamma' & \trans\bfupsilon'\end{pmatrix} \begin{pmatrix}\one_g & \\ & p^{-1}\one_g\end{pmatrix}\begin{pmatrix}\bfgamma\\ \bfupsilon\end{pmatrix}\right)\quad d\mu\right)}.
\end{align*}
Notice that \begin{align*}
    e_{\kappa}^{\hst}\left(\begin{pmatrix}\trans\bfgamma' & \trans\bfupsilon'\end{pmatrix} \begin{pmatrix}\one_g & \\ & p^{-1}\one_g\end{pmatrix}\begin{pmatrix}\bfgamma\\ \bfupsilon\end{pmatrix}\right) = e_{\kappa}^{\hst}(\trans\bfgamma'\bfgamma + \trans\bfupsilon'\bfupsilon/p)
\end{align*} is valid since $(\trans\bfgamma'\bfgamma + \trans\bfupsilon'\bfupsilon/p) \in \Iw_{\GL_g}^+$ (this is because $\Iw_{\GL_g}^+$ is stable under transpose and both $\bfupsilon$ and $\bfupsilon'$ are divisible by $p$).\\

Consequently, we have the pairing \[
    \bls\cdot, \cdot\brs_{\kappa}^{\circ}: D_{\kappa}^r(\T_0, R)\times D_{\kappa}^r(\T_0, R) \rightarrow R
\] given by the formula \[
    [\mu_1, \mu_2]_{\kappa}^{\circ} = \scalemath{1}{\int_{\T_{00}^2} e_{\kappa}^{\hst}\left(\begin{pmatrix}\trans\bfgamma_2 & \trans\bfupsilon_2\end{pmatrix} \begin{pmatrix}\one_g & \\ & p^{-1}\one_g\end{pmatrix}\begin{pmatrix}\bfgamma_1\\ \bfupsilon_1\end{pmatrix}\right)\quad d\mu_1(\bfgamma_1, \bfupsilon_1)d\mu_2(\bfgamma_2, \bfupsilon_2)}
\]

\begin{Proposition}\label{Proposition: main property of the pairing}
For any $\bfalpha = \begin{pmatrix} \bfalpha_a & \bfalpha_b\\ \bfalpha_c & \bfalpha_d\end{pmatrix} \in \Xi$, write $$\bfalpha^{\Shi} = \begin{pmatrix} \trans\bfalpha_a & \trans\bfalpha_c/p\\ p\trans\bfalpha_b & \trans\bfalpha_d\end{pmatrix}\in \Xi.$$ Then, for any $\mu_1, \mu_2\in D_{\kappa}^r(\T_0, R)$, we have $$\bls \bfalpha \cdot \mu_1, \mu_2\brs_{\kappa}^{\circ}  = \bls \mu_1, \bfalpha^{\Shi}\cdot \mu_2\brs_{\kappa}^{\circ}.$$
\end{Proposition}
\begin{proof}
The assertion follows from the computation \begin{align*}
    \begin{pmatrix}\trans\bfgamma_2 & \trans\bfupsilon_2\end{pmatrix} &  \begin{pmatrix}\one_g & \\ & p^{-1}\one_g\end{pmatrix} \begin{pmatrix} \bfalpha_a & \bfalpha_b\\ \bfalpha_c & \bfalpha_d\end{pmatrix} \begin{pmatrix}\bfgamma_1\\ \bfupsilon_1\end{pmatrix} \\
    & = \begin{pmatrix}\trans\bfgamma_2 & \trans\bfupsilon_2\end{pmatrix} \begin{pmatrix} \bfalpha_a & p\bfalpha_b\\ \bfalpha_c/p & \bfalpha_d\end{pmatrix} \begin{pmatrix}\one_g & \\ & p^{-1}\one_g\end{pmatrix} \begin{pmatrix}\bfgamma_1\\ \bfupsilon_1\end{pmatrix} \\ 
    & = \trans\left(\begin{pmatrix} \trans\bfalpha_a & \trans\bfalpha_c/p\\ p\trans\bfalpha_b & \trans\bfalpha_d\end{pmatrix}\begin{pmatrix} \bfgamma_2 \\ \bfupsilon_2\end{pmatrix}\right) \begin{pmatrix}\one_g & \\ & p^{-1}\one_g\end{pmatrix} \begin{pmatrix}\bfgamma_1\\ \bfupsilon_1\end{pmatrix}.
\end{align*}
\end{proof}

\begin{Remark}\label{Remark: relation to g=1}
\normalfont When comparing with our algebraic model, one notices that the definition of the pairing $\bls\cdot, \cdot\brs_{\kappa}^{\circ}$ involves a ``\textit{normalisation}'' by $p^{-1}$. Such a normalisation is due to our model for $g=1$. More precisely, when $g=1$, elements in $\T_0$ can be written as $(1, pc) a $ for some $a\in \Z_p^\times$ and $c\in \Z_p$. Then, for any $\mu_1, \mu_2\in D_{\kappa}^r(\T_0, R)$, we have \begin{align*}
    \bls \mu_1, \mu_2\brs_{\kappa}^{\circ} = \int_{\T_{00}^2}\kappa(1+pc_1c_2) \quad d\mu_1(1, c_1)d\mu_2(1, c_2),
\end{align*} which then coincides with the interpretation in Hansen's unpublished notes \cite{Hansen-notes}. In particular, by applying \cite[Definition VIII.2.4]{Bellaiche-eigenbook}, we have the formula 
\[
\bls\mu_1, \mu_2\brs_{\kappa}^{\circ} = \sum_{i=0}^{\infty} p^i \left(\substack{\kappa\\ i}\right) \mu_1(c_1^i)\mu_2(c_2^i),
\] which is (almost) the same formula given by [\textit{op. cit.}, (VIII.2.4)]. Here, for $j=1, 2$, we view $c_j^i$ as a function on $\T_0$ via $$c_j^i: \T_0\ni (a, pc)\mapsto \kappa(a)(c/a)^i.$$
\end{Remark}

\begin{Remark}
\normalfont Following Remark \ref{Remark: symplecti pairing}, for any dominant $k\in \Z_{>0}^g$, we may consider the pairing $\bls\cdot, \cdot\brs_{k}^{\circ}$ to be the twist of $\bla\cdot, \cdot\bra_k$ by an \textit{Atkin--Lehner operator}. More precisely, let $$\w_p := \begin{pmatrix} & -p^{-1}\oneanti_g\\ \oneanti_g\end{pmatrix},$$ then \begin{align*}
    \bls\mu_1, \mu_2\brs_{k}^{\circ} & = \scalemath{1}{\int_{\T_{00}^2}e_{k}^{\hst}\left(\begin{pmatrix}\trans\bfgamma_2 & \trans\bfupsilon_2\end{pmatrix}\begin{pmatrix}\one_g\\ & p^{-1}\one_g\end{pmatrix}\begin{pmatrix}\bfgamma_1\\ \bfupsilon_1\end{pmatrix}\right) \quad d\mu_1(\bfgamma_1, \bfupsilon_1)d\mu_2(\bfgamma_2, \bfupsilon_2)}\\
    & = \scalemath{1}{\int_{\T_{00}^2}e_{k}^{\hst}\left( \begin{pmatrix}\trans\bfgamma_2 & \trans\bfupsilon_2\end{pmatrix}\begin{pmatrix}& \oneanti_g\\ -p^{-1}\oneanti_g\end{pmatrix}\begin{pmatrix} & -\oneanti_g\\ \oneanti_g\end{pmatrix}\begin{pmatrix}\bfgamma_1\\ \bfupsilon_1\end{pmatrix}\right)  \quad d\mu_1(\bfgamma_1, \bfupsilon_1)d\mu_2(\bfgamma_2, \bfupsilon_2)}\\
    & = \scalemath{1}{\int_{\T_{00}^2}e_{k}^{\hst}\left(\trans\left(\w_p\begin{pmatrix}\bfgamma_2\\ \bfupsilon_2\end{pmatrix}\right)\begin{pmatrix} & -\oneanti_g\\ \oneanti_g\end{pmatrix}\begin{pmatrix}\bfgamma_1\\ \bfupsilon_1\end{pmatrix}\right)  \quad d\mu_1(\bfgamma_1, \bfupsilon_1)d\mu_2(\bfgamma_2, \bfupsilon_2).}
\end{align*} In particular, this viewpoint coincides with the perspectives in \cite{Kim, Bellaiche-eigenbook, Hansen-notes} when $g=1$.
\end{Remark}

%% file: 3-Eigenvariety.tex
\section{The overconvergent cohomology and the eigenvariety}
\subsection{A pairing on cohomology groups}
Let $N\in \Z_{>3}$ such that $p\nmid N$ and we fix an $N$-th primitive root of unity $\zeta_N$ (and so we fixed an isomorphism $\mu_N\simeq \Z/N\Z$). Equip on $(\Z/N\Z)^{2g}$ a symplectic pairing induced by the pairing on $\V_{\Z}$ in the previous section. Let $\Sch_{\Z_p[\zeta_N]}$ be the category of locally noetherian schemes over $\Z_p[\zeta_N]$, then the moduli problem \begin{align*}
    \Sch_{\Z_p[\zeta_N]} & \rightarrow \Sets,\\
    S & \mapsto \scalemath{1}{\left\{(A_{/S}, \lambda, \alpha_N): \begin{array}{l}
    A_{/S}\text{ is a principally polarised abelian scheme over }S  \\
    \lambda: A\rightarrow A^{\vee}\text{ is the principal polarisation}\\
    \alpha_N: A[N]\xrightarrow{\sim}(\Z/N\Z)^{2g}\text{ is a symplectic isomorphism}
\end{array}\right\}/\simeq}
\end{align*} is representable by a scheme $X_{\Z_p[\zeta_N]}$, where the symplectic isomorphism $\alpha_N$ is taken with respect to the Weil pairing on $A[N]$ and the pairing induced by $\bla \cdot, \cdot\bra$ on $(\Z/N\Z)^{2g}$. We let $$X=X_{\C_p}:= X_{\Z_p[\zeta_N]}\times_{\Z_p[\zeta_N]} \Spec \C_p.$$\\

Let $\zeta_p$ be a primitive $p$-th root of unity. We also consider the scheme $X_{\Iw^+, \Q_p[\zeta_p, \zeta_N]}$, paramatrising the isomorphism classes of tuples $(A, \lambda, \alpha_N, \Fil_{\bullet}A[p], \{C_i: i=1, ..., g\})$, where \begin{enumerate}
    \item[$\bullet$]  $(A, \lambda, \alpha_N)\in X_{\Q_p[\zeta_p, \zeta_N]} = X_{\Z_p[\zeta_N]}\times_{\Z_p[\zeta_N]}\Spec\Q_p[\zeta_p, \zeta_N]$,
    \item[$\bullet$] $\Fil_{\bullet}A[p]$ is a full flag of $A[p]$ such that \[
        (\Fil_{\bullet}A[p])^{\perp} \simeq \Fil_{2g-\bullet}A[p]
    \] with respect to the Weil pairing, and
    \item[$\bullet$] $\{C_i: i=1, ..., g\}$ is a collection of subgroups of $A[p]$ of order $p$ such that $\Fil_i A[p] = \langle C_1, ..., C_i\rangle$ for $i=1, ..., g$.
\end{enumerate} Again, we write $X_{\Iw^+}:=X_{\Iw^+, \Q_p[\zeta_p, \zeta_N]}\times_{\Q_p[\zeta_p, \zeta_N]}\Spec \C_p$. Obviously, we have a natural forgetful map $$\pi_{\Iw^+}: X_{\Iw^+}\rightarrow X, \quad (A, \lambda, \alpha_N, \Fil_{\bullet}A[p], \{C_i:i=1, ..., g\})\mapsto (A, \lambda, \alpha_N).$$ \\

Via the fixed (algebraic) isomorphism $\C_p\simeq \C$, we consider the locally symmetric space $X_{\Iw^+}(\C)$ in the rest of this article, which admits an alternative description $$X_{\Iw^+}(\C)=\GSp_{2g}(\Q)\backslash \GSp_{2g}(\A_f)\times \bbH_g/\Iw_{\GSp_{2g}}^+\Gamma(N),$$ where \begin{enumerate}
    \item[$\bullet$] $\A_f$ is the ring of finite ad\`{e}les of $\Q$, 
    \item[$\bullet$] $\bbH_g$ is the disjoint union of the Siegel upper- and lower-half spaces of genus $g$, 
    \item[$\bullet$] $\Gamma(N):=\{\bfgamma\in \GSp_{2g}(\widehat{\Z}): \bfgamma\equiv \one_{2g}\mod N\}$. 
    %\item[$\bullet$] $\SpO_{2g}(\R):=\{\bfgamma\in \Sp_{2g}(\R): \trans\bfgamma=\bfgamma^{-1}\}$.
\end{enumerate}

\vspace{3mm}

Fix a $p$-adic weight $\kappa: T_{\GL_g, 0}\rightarrow R^{\times}$ satisfying the assumptions on the Banach--Tate $\Z_p$-algebra norm $|\cdot|_R$ on $R$ together with a fixed  multiplicative pseudouniformiser $\varpi\in R$. Recall the analytic distributions $D_{\kappa}^r(\T_0, R)$ that we introduced in the previous section. From now on, we simplify the notation by writing $$D_{\kappa}^r = D_{\kappa}^r(\T_0, R).$$\\

Since $D_{\kappa}^r$ admits a left $\Iw_{\GSp_{2g}}^+$-action, we can follow the strategy in \cite{Hansen-PhD} to compute the Betti cohomology groups $H^t(X_{\Iw^+}(\C), D_{\kappa}^{r})$. That is, we consider the so-called \textit{Borel--Serre cochain complex} $C^{\bullet}(\Iw_{\GSp_{2g}}^+, D_{\kappa}^r)$, constructed by fixing a triangulation on the Borel--Serre compactification $\overline{X}_{\Iw^+}^{\BS}(\C)$ of the locally symmetric space $X_{\Iw^+}(\C)$ (see \cite{Borel-Serre}). The Borel--Serre cochain complex admits the following nice properties (see also \cite[\S 2.1]{Hansen-PhD}):\begin{enumerate}
    \item There is a homotopy between the singular cochain complex and the Borel--Serre cochain complex and hence the reason why one can compute the cohomology groups by considering the Borel--Serre cochain complex. 
    \item The total space $C_{\kappa, r}^{\tol}:=\oplus_{t}C^t(\Iw_{\GSp_{2g}}^+, D_{\kappa}^r)$ is a potentially ON-able Banach module over $R$ since $C^{\bullet}(\Iw_{\GSp_{2g}}^+, D_{\kappa}^r)$ is a finite cochain complex and $D_{\kappa}^r(\T_0, R)$ is potentially ON-able with an explicit potential ON-basis described in \cite[\S 3.2]{Johansson-Newton}. 
\end{enumerate}

The fixed triangulation on $\overline{X}_{\Iw^+}^{\BS}(\C)$ provides also a triangulation on the boundary $\partial \overline{X}_{\Iw^+}^{\BS}(\C):=\overline{X}_{\Iw^+}^{\BS}(\C)\smallsetminus X_{\Iw^+}(\C)$ and hence defines a cochain complex $C^{\bullet}_{\partial}(\Iw_{\GSp_{2g}}^+, D_{\kappa}^r)$ that computes the cohomology groups at the boundary. The natural closed embedding $\partial\overline{X}_{\Iw^+}^{\BS}(\C)\hookrightarrow \overline{X}^{\BS}_{\Iw^+}(\C)$ then induces a morphism of cochain complexes $$\pi: C^{\bullet}(\Iw_{\GSp_{2g}}^+, D_{\kappa}^r)\rightarrow C^{\bullet}_{\partial}(\Iw_{\GSp_{2g}}^+, D_{\kappa}^r).$$ Following \cite[\S 3.1.3]{Barrera}, we define $C_c^{\bullet}(\Iw_{\GSp_{2g}}^+, D_{\kappa}^r):=\Cone(\pi)$ the mapping cone of $\pi$, \textit{i.e.}, \begin{align*}
    & \Cone(\pi)^{t}=C^t(\Iw_{\GSp_{2g}}^+, D_{\kappa}^r)\oplus C_{\partial}^{t-1}(\Iw_{\GSp_{2g}}^+, D_{\kappa}^r)\text{ with }\\
    & d_c^t: \Cone(\pi)^t\rightarrow \Cone(\pi)^{t+1},\quad (\sigma, \sigma_{\partial})\mapsto (-d^t\sigma, -\pi^i\sigma+d_{\partial}^{t-1}\sigma_{\partial}),
\end{align*} where $d$ and $d_{\partial}$ are differentials on $C^{\bullet}(\Iw_{\GSp_{2g}}^+, D_{\kappa}^r)$ and $C^{\bullet}(\Iw_{\GSp_{2g}}^+, D_{\kappa}^r)$ respectively. The strategy of the proof of \cite[Proposition 3.5]{Barrera} applies here and one sees that $C_c^{\bullet}(\Iw_{\GSp_{2g}}^+, D_{\kappa}^r)$ computes the compactly supported cohomology groups $H^t_c(X_{\Iw^+}(\C), D_{\kappa}^r)$. Moreover, the natural morphism \[
    C_c^{\bullet}(\Iw_{\GSp_{2g}}^+, D_{\kappa}^r)\rightarrow C^{\bullet}(\Iw_{\GSp_{2g}}^+, D_{\kappa}^r)
\] induces a morphism on the cohomology groups $$H^t_c(X_{\Iw^+}(\C), D_{\kappa}^r)\rightarrow H^t(X_{\Iw^+}(\C), D_{\kappa}^r).$$ For each $t$, we let $$H_{\Par}^t(X_{\Iw^+}(\C), D_{\kappa}^r):=\image\left(H^t_c(X_{\Iw^+}(\C), D_{\kappa}^r)\rightarrow H^t(X_{\Iw^+}(\C), D_{\kappa}^r)\right),$$ and call them the parabolic cohomology groups. 

\begin{Proposition}\label{Proposition: well-defined pairing on the parabolic cohomology}
Let $n_0=g(g+1)/2$ be the $\C$-dimension of $X_{\Iw^+}(\C)$. Then, we have a well-defined pairing 
\[
\bls\cdot, \cdot\brs_{\kappa}: H^t_{\Par}(X_{\Iw^+}(\C), D_{\kappa}^r)\times H_{\Par}^{2n_0-t}(X_{\Iw^+}(\C), D_{\kappa}^r)\rightarrow R
\] for any $0\leq t\leq 2n_0$.
\end{Proposition}
\begin{proof}
Recall the pairing $\bls\cdot, \cdot\brs_{\kappa}^{\circ}$ defined in \S \ref{subsection: pairing on the analytic distributions}. Together with the cup product on cohomology groups, one obtains a pairing $\bls\cdot, \cdot\brs_{\kappa}^*$ defined as the composition 
\[
\begin{tikzcd}
    \scalemath{1}{H^t_c(X_{\Iw^+}(\C), D_{\kappa}^r) \times H^{2n_0-t}(X_{\Iw^+}(\C), D_{\kappa}^r)} \arrow[r, "\smile"]\arrow[rdd, bend right = 15,  "\text{$\bls\cdot, \cdot\brs_{\kappa}^*$}"'] &  \scalemath{1}{H^{2n_0}_{c}(X_{\Iw^+}(\C), D_{\kappa}^r\widehat{\otimes}_R D_{\kappa}^r)}\arrow[d, "\text{$\bls\cdot, \cdot\brs_{\kappa}^{\circ}$}"]\\ & \scalemath{1}{H_c^{2n_0}(X_{\Iw^+}(\C), R)}\arrow[d, "\simeq"] \\ & R,
\end{tikzcd}
\] where ``$\smile$'' denotes the cup product. \\

The compatibility of cup products (see, for example, \cite[Chapter 5, \S 48, Exercise 2]{Munkers84}) yields the commutative diagram $$\begin{tikzcd}[column sep = small]
H_c^t(X_{\Iw^+}(\C), D_{\kappa}^r)\times H^{2n_0-t}(X_{\Iw^+}(\C), D_{\kappa}^r)\arrow[r, "\smile"] & H_c^{2n_0}(X_{\Iw^+}(\C), D_{\kappa}^r\widehat{\otimes}_R D_{\kappa}^r)\arrow[d, equal]\\
H_c^{t}(X_{\Iw^+}(\C), D_{\kappa}^r)\times H_c^{2n_0-t}(X_{\Iw^+}(\C), D_{\kappa}^r)\arrow[r, "\smile"]\arrow[u, shift left = 15, equal]\arrow[u, shift right = 15]\arrow[d, shift right = 15]\arrow[d, shift left = 15, equal] & H_c^{2n_0}(X_{\Iw^+}(\C), D_{\kappa}^r\widehat{\otimes}_R D_{\kappa}^r)\arrow[d, equal]\\
H^t(X_{\Iw^+}(\C), D_{\kappa}^r)\times H_c^{2n_0-t}(X_{\Iw^+}(\C), D_{\kappa}^r)\arrow[r, "\smile"] & H_c^{2n_0}(X_{\Iw^+}(\C), D_{\kappa}^r\widehat{\otimes}_R D_{\kappa}^r)
\end{tikzcd}.$$ In particular, if $[\mu_1]\in H_{\Par}^t(X_{\Iw^+}(\C), D_{\kappa}^r)$ and $[\mu_2]\in H_{\Par}^{2n_0-t}(X_{\Iw^+}(\C), D_{\kappa}^r)$ with $[\mu_1']\in H_c^t(X_{\Iw^+}(\C), D_{\kappa}^r)$ and $[\mu_2']\in H_c^{2n_0-t}(X_{\Iw^+}(\C), D_{\kappa}^r)$ such that $[\mu_i']\mapsto [\mu_i]$ for $i=1, 2$, then $$[\mu_1]\smile [\mu_2']=[\mu_1']\smile[\mu_2']=[\mu_1']\smile [\mu_2].$$ Hence we define $$\bls [\mu_1], [\mu_2]\brs_{\kappa}=\bls [\mu_1'], [\mu_2]\brs_{\kappa}^*=\bls [\mu_1], [\mu_2']\brs_{\kappa}^*.$$ We see that $\bls\cdot, \cdot\brs_{\kappa}$ is well-defined, \textit{i.e.}, independent of the choice of the lifting, due to the commutativity of the above diagram.
\end{proof}

\subsection{Hecke operators}

\paragraph{Hecke operators outside $pN$.} Let $q$ be a prime number not dividing $pN$. We consider the set of double cosets $$\Upsilon_q:=\{[\GSp_{2g}(\Z_q)\bfdelta\GSp_{2g}(\Z_q)]: \bfdelta\in \GSp_{2g}(\Q_q)\cap M_{2g}(\Z_q)\}.$$ For any fixed $\bfdelta$, we have the coset decomposition $$\GSp_{2g}(\Z_q)\bfdelta\GSp_{2g}(\Z_q)=\sqcup_{j} \bfdelta_j\GSp_{2g}(\Z_p)$$ for finitely many $\bfdelta_j\in \GSp_{2g}(\Q_q)\cap M_{2g}(\Z_q)$. By letting $\bfdelta_j$'s act trivially on $D_{\kappa}^r$, we have a left action of the double coset $[\GSp_{2g}(\Z_q)\bfdelta\GSp_{2g}(\Z_q)]$ on $C^{\bullet}(\Iw_{\GSp_{2g}}^+, D_{\kappa}^r)$ by $$[\GSp_{2g}(\Z_q)\bfdelta\GSp_{2g}(\Z_q)]\cdot \sigma = \sum_{j}\bfdelta_j\cdot \sigma$$ for any $\sigma\in C^{\bullet}(\Iw_{\GSp_{2g}}^+, D_{\kappa}^r)$. Then the Hecke algebra at $q$ (over $\Z_p$) is defined to be $\bbT_q=\bbT_{q, \Z_p}=\Z_p[\Upsilon_q]$.

\paragraph{Hecke operators at $N$.} We ignore the Hecke actions at $N$, \textit{i.e.}, for $\ell|N$, we only consider the trivial action and hence the Hecke algebra at $\ell$ is $\bbT_{\ell}=\bbT_{\ell, \Z_p}:=\Z_p$.

\paragraph{Hecke operator at $p$.} Let \begin{align*}
    & \bfu_{p,0}:=\begin{pmatrix}\one_g\\ & p\one_g\end{pmatrix}\\
    & \bfu_{p, i}:=\begin{pmatrix} \one_{g-i} \\ & p\one_{i}\\ & & p\one_{i}\\ & & & p^2\one_{g-i}\end{pmatrix}\in T_{\GSp_{2g}}(\Q_p)\cap M_{2g}(\Z_p)\text{ for }1\leq i\leq g-1\\
    & \bfu_p:= \prod_{i=0}^{g-1}\bfu_{p, i}%=\begin{pmatrix}\diag(1, p, ..., p^{g-1})\\ & \diag(p^{2g-1}, p^{2g-2}, ..., p^{g})\end{pmatrix}
\end{align*} and consider the set of double cosets $$\Upsilon_p:=\{[\Iw_{\GSp_{2g}}^+\bfu_{p, i}\Iw_{\GSp_{2g}}^+]: i=0, ..., g-1\}.$$ We immediately see that $[\Iw_{\GSp_{2g}}^+\bfu_p\Iw_{\GSp_{2g}}^+]=\prod_{i=0}^{g-1}[\Iw_{\GSp_{2g}}^+\bfu_{p, i}\Iw_{\GSp_{2g}}^+]$. A direct computation shows that the coset decomposition of $\Iw_{\GSp_{2g}}^+\bfu_{p,i}\Iw_{\GSp_{2g}}^+$ can be given by $$\Iw_{\GSp_{2g}}^+\bfu_{p, i}\Iw_{\GSp_{2g}}^+=\sqcup_{j}\bfdelta_{i, j}\Iw_{\GSp_{2g}}^+$$ for some $\bfdelta_{i, j}\in \GSp_{2g}(\Q_p)\cap M_{2g}(\Z_p)$; in particular, $\bfdelta_{i,j} = \bflambda_{i,j}\bfu_{p, i}$ for some $\bflambda_{i,j}\in \Iw_{\GSp_{2g}}^+$. \\

For any $(\bfgamma, \bfupsilon)\in \T_0$, write $(\bfgamma, \bfupsilon) = (\bfgamma_0, \bfupsilon_0)\bfbeta$ for some $\bfbeta\in B_{\GL_g, 0}^+$ such that $\bfgamma_0\in U_{\GL_g, 1}^{\opp}$. Then, the left action of $\bfu_{p,i}$ on $\T_0$ is defined by the formula $$\bfu_{p,i}\cdot (\bfgamma, \bfupsilon) = (\bfu_{p, i}^{\square}\bfgamma_0\bfu_{p, i}^{\square, -1}, \bfu_{p,i}^{\blacksquare}\bfupsilon_{0}\bfu_{p, i}^{\square, -1})\bfbeta,$$ where we write $$\bfu_{p, i} = \begin{pmatrix}\bfu_{p, i}^{\square} & \\ & \bfu_{p,i}^{\blacksquare}\end{pmatrix}.$$ Consequently, this defines a left action of $\bfu_{p,i}$ on $D_{\kappa}^r$. On the other hand, the right multiplication of $\bfdelta_{i, j}$ on $\GSp_{2g}(\A_f)$ gives a left action on the homomorphisms between the free abelian group of simplical complexes on $\overline{X}^{\BS}_{\Iw}(\C)$. The two actions then combine to an action on $C^{\bullet}(\Iw_{\GSp_{2g}}^+, D_{\kappa}^r)$ and hence defines the action of $[\Iw_{\GSp_{2g}}^+ \bfu_{p,i} \Iw_{\GSp_{2g}}^+]$ on $C^{\bullet}(\Iw_{\GSp_{2g}}^+, D_{\kappa}^r)$ by 
\[
[\Iw_{\GSp_{2g}}^+ \bfu_{p,i} \Iw_{\GSp_{2g}}^+] \cdot \sigma : = \sum_{j}\bfdelta_{i,j}\sigma = \sum_{j} \bflambda_{i,j}\cdot \left(\bfu_{p, i}\cdot \sigma\right)
\] for any $\sigma\in C^{\bullet}(\Iw_{\GSp_{2g}}^+, D_{\kappa}^r)$. We shall denote by $U_{p,i}$ and $U_p$ the operators on $C^{\bullet}(\Iw_{\GSp_{2g}}^+, D_{\kappa}^r)$ corresponding to the classes of double cosets $[\Iw_{\GSp_{2g}}^+\bfu_{p, i}\Iw_{\GSp_{2g}}^+]$ and $[\Iw_{\GSp_{2g}}^+ \bfu_p\Iw_{\GSp_{2g}}^+]$. The Hecke algebra at $p$ is then defined to be $\bbT_p=\bbT_{p, \Z_p}=\Z_p[\Upsilon_p]$.

\begin{Definition}\label{Definition: Hecke algebra}
We define \begin{align*}
    &\bbT^p  :=\otimes_{q\neq p}\bbT_q=\textbf{ the Hecke algebra outside $p$}\\
    &\bbT  := \bbT^p\otimes_{\Z_p}\bbT_p=\textbf{ the total Hecke algebra}.
\end{align*}
\end{Definition}

\begin{Lemma}\label{Lemma: the parabolic cohomolgy is Hecke stable}
The parabolic cohomology groups $H_{\Par}^r(X_{\Iw^+}(\C), D_{\kappa}^r)$ are $\bbT$-stable.
\end{Lemma}
\begin{proof}
Due to the nature of the Borel--Serre compactification, $C_{\partial}^{\bullet}(\Iw_{\GSp_{2g}^+}, D_{\kappa}^r)$ admits Hecke actions as the ones defined above. Hence \[
    \pi: C^{\bullet}(\Iw_{\GSp_{2g}}^+, D_{\kappa}^r)\rightarrow C_{\partial}^{\bullet}(\Iw_{\GSp_{2g}}^+, D_{\kappa}^r)
\] is a Hecke equivariant morphism of cochain complexes and hence $$C_c^{\bullet}(\Iw_{\GSp_{2g}^+}, D_{\kappa}^r)\rightarrow C^{\bullet}(\Iw_{\GSp_{2g}^+}, D_{\kappa}^r)$$ is also Hecke equivariant and induces a Hecke equivariant map on cohomology groups $$H^t_c(X_{\Iw^+}(\C), D_{\kappa}^r)\rightarrow H^t(X_{\Iw^+}(\C), D_{\kappa}^r).$$ This then shows the desired result.
\end{proof}

\subsection{The cuspidal eigenvariety}\label{subsection: cuspidal eigenvariety} In this subsection, we extract out the cuspidal part of the eigenvarieties constructed in \cite{Johansson-Newton}. Although this is an easy consequence of \textit{op. cit.} for experts, we write down the construction after recalling sufficiently amount of materials. 

\begin{Lemma}\label{Lemma: representation of the weight functor}
The functor assigning each sheafy complete affinoid $(\Z_p, \Z_p)$-algebra $(R, R^+)$ to the set $\Hom_{\cts}(T_{\GL_g, 0}, R^\times)$ is represented by the affinoid algebra $(\Z_p\llbrack T_{\GL_g, 0}\rrbrack, \Z_p\llbrack T_{\GL_g, 0}\rrbrack)$.
\end{Lemma}
\begin{proof}
For any sheafy complete affinoid $(\Z_p, \Z_p)$-algebra $(R, R^+)$, we have a bijection $$\Hom_{\cts}(T_{\GL_g,0}, R^\times) \simeq \Hom_{\Z_p}^{\cts}(\Z_p\llbrack T_{\GL_g, 0}\rrbrack, R).$$ The bijection is obtained by extending the characters on the left-hand-side $\Z_p$-linearly. Note that any continuous character $T_{\GL_g, 0}\rightarrow R$ automatically lands in $(A^+)^\times$ as discussed in \cite[Defintiion 4.1]{Johansson-Newton}.
\end{proof}

\begin{Definition}\label{Definition: the weight space}
The \textbf{weight space} in our concern is then defined to be $$\calW:=\Spa(\Z_p\llbrack T_{\GL_g, 0}\rrbrack, \Z_p\llbrack T_{\GL_g, 0}\rrbrack)^{\an},$$ where the superscript ``$\bullet^{\an}$'' means that we are taking the analytic locus of the corresponding adic space. 
\end{Definition}

For any open affinoid $\calU\subset \calW$, we will always write the corresponding affinoid algebra to be $(R_{\calU}, R_{\calU}^+)$ and the universal weight on $\calU$ to be $\kappa_{\calU}: T_{\GL_g, 0}\rightarrow R_{\calU}^{\times}$. In the following, we will always assume that the universal weight $\kappa_{\calU}$ and $R_{\calU}$ admit the previous assumptions that we made for $p$-adic weights. That is, 
we assume that one can choose a norm $|\cdot|_{R_{\calU}}$ on $R_{\calU}$ so that $R_{\calU}$ is a Banach--Tate $\Z_p$-algebra and that $|\cdot|_{R_{\calU}}$ is adapted to $\kappa_{\calU}$.\\

Let $\bbA_{\Z_p}^{1, \ad}:=\Spa(\Z_p[T], \Z_p)$, we write $\bbA_{\calU}^1:=\calU\times_{\Spa(\Z_p, \Z_p)}\bbA_{\Z_p}^{1, \ad}$ for any open affinoid $\calU\subset\calW$. We have the following explicit description $$\bbA_{\calU}^1=\cup_m \Spa\left(R_{\calU}\langle \varpi^mT\rangle, R_{\calU}^+\langle \varpi^m T\rangle\right),$$ where $\varpi$ is a fixed pseudouniformiser of $R_{\calU}$. Moreover, the global functions on $\bbA^1_{\calU}$ is the ring $$R_{\calU}\{\{ T\}\}:=\left\{\sum_{n\geq 0}a_n T^n\in R_{\calU}\llbrack T\rrbrack: |a_n|_R M^n\rightarrow 0\text{ for all }M\in \R_{\geq 0}\right\}.$$\\

Fix an open affinoid $\calU\subset\calW$, recall that the total Borel--Serre cochain complex $C^{\tol}_{\kappa_{\calU}, r}$ is a potentially ON-able Banach $R_{\calU}$-module. Moreover, $U_p$ acts on $C^{\tol}_{\kappa_{\calU}, r}$ compactly by \cite[Corollary 3.3.10]{Johansson-Newton}, hence we can consider the Fredholm determinant $$F_{\kappa_{\calU}}^r(T):= \det\left(1-TU_p | C^{\tol}_{\kappa_{\calU}, r}\right)\in R_{\calU}\{\{T\}\}.$$ According to [\textit{op.cit.}, Proposition 4.1.2 \& Proposition 4.1.4], the Fredholm determinant doesn't depend on $r\in [r_{\kappa},1)$ and the chosen norm on $R_{\calU}$, thus we write $F_{\kappa_{\calU}}$. By [\textit{op. cit.}, Corollary 4.1.5], the Fredholm determinants $(F_{\kappa_{\calU}})_{\calU}$, where $\calU$ ranges over all open affinoid $\calU\subset\calW$, glue together to $F_{\calW}\in \scrO_{\calW}(\calW)\{\{T\}\}$ and hence we define the \textit{\textbf{Fredholm hypersurface}} (or the \textit{\textbf{spectral variety}}) $$\calZ:= \text{ the zero locus of $F_{\calW}$ in $\bbA^1_{\calW}$}$$ and denote by $\wt_{\calZ}: \calZ\rightarrow \calW$ the structure morphism.

\begin{Definition}\label{Definition: slope datum}
Let $h=\frac{m}{n}\in \Q_{\geq 0}$ and define $$\bbB_{\calU, h}:=\{x\in \bbA^1_{\calU}: |T^n|_x\leq |\varpi^{-m}|_x\}.$$ We also define $\calZ_{\calU}$ to be the zero locus of $F_{\kappa_{\calU}}$ in $\bbA_{\calU}^1$. Then, we say the pair $(\calU, h)$ is a \textbf{slope datum} if and only if $$\calZ_{\calU, h}:=\calZ_{\calU}\cap \bbB_{\calU, h}\rightarrow \calU$$ is finite of constant degree. 
\end{Definition}

\begin{Proposition}[$\text{\cite[Theorem 2.3.2]{Johansson-Newton}}$]\label{Proposition: properties of the slope decomposition of the Fredholm determinant}
Keep the above notations. We have the following\begin{enumerate}
    \item The pair $(\calU, h)$ is a slope data if and only if $F_{\kappa_{\calU}}$ admits a factorisation $F_{\kappa_{\calU}}=QS$, where\begin{enumerate}
        \item[$\bullet$] $Q$ is a polynomial whose leading coefficient is a unit in $R_{\kappa_{\calU}}$ and its corresponding Newton polygon has slope $\leq h$ (see \cite[Definition 2.2.4]{Johansson-Newton}),
        \item[$\bullet$] $S=1+\sum_{n>0}a_n T^n\in R_{\kappa_{\calU}}\{\{T\}\}$ and 
        \item[$\bullet$] the ideal generated by $Q$ and $S$ in $R_{\kappa_{\calU}}\{\{T\}\}$ it the unit ideal.
    \end{enumerate}
    \item The collection $\Cov_{\sd}(\calZ):=\{\calZ_{\calU, h}: (\calU, h)\text{ is a slope datum}\}$ is an open cover for $\calZ$.
\end{enumerate}
\end{Proposition}

\begin{Remark}\label{Remark: Fredholm hypersurfacr for compactly supported cohomology}
\normalfont Recall that the cochain complex $C_c^{\bullet}(\Iw_{\GSp_{2g}}^+, D_{\kappa}^r)$ computes the compactly supported cohomology groups. By definition, the total complex $C_{c, \kappa,r}^{\tol}:=\oplus_t C^t_c(\Iw_{\GSp_{2g}}^+, D_{\kappa}^r)$ is a finite number of copies of $D_{\kappa}^r(\T_0, R)$ as an $R$-module. Therefore, it is a potentially ON-able module and the above also applies to $C_{c, \kappa,r}^{\tol}$. In particular, we have a Fredholm hypersurface $\calZ^c$ when considering the compactly supported cohomology groups and for any slope datum $(\calU, h)$ for $\calZ^c$, we have a slope decomposition $C_{c, \kappa_{\calU}, r}^{\tol}=C_{c, \kappa_{\calU}, r}^{\tol, \leq h}\oplus C_{c, \kappa_{\calU}, r}^{\tol, >h}$.
\end{Remark}

Let $(\calU, h)$ be a slope datum for $\calZ$ and let $C_{\kappa_{\calU}}^{\tol}=\oplus_t C^t(\Iw_{\GSp_{2g}}^+, D_{\kappa_{\calU}}^{\dagger})$. As the above discussions hold for all $r\in [r_{\kappa}, 1)$, the results also apply to $C_{\kappa_{\calU}}^{\tol}$. In particular, when considering the $U_p$-operator acting on $C_{\kappa_{\calU}}^{\tol}$, we have the factorisation of the corresponding Fredholm determinant $F_{\kappa_{\calU}}$ and the decomposition $C_{\kappa_{\calU}}^{\tol}=C_{\kappa_{\calU}}^{\tol, \leq h}\oplus C_{\kappa_{\calU}}^{\tol, >h}$. Define $H^t(X_{\Iw^+}(\C), D_{\kappa_{\calU}}^{\dagger})^{\leq h}$ to be the $t$-th cohomology group of the cochain complex $C_{\kappa_{\calU}}^{\tol, \leq h}$ and let $H_{\kappa_{\calU}}^{\tol,\leq h}=\oplus_{t} H^t(X_{\Iw^+}(\C), D_{\kappa_{\calU}}^{\dagger})^{\leq h}$. From the construction of the eigenvarieties in \cite{Johansson-Newton}, we know that the assignment $$\Cov_{\sd}(\calZ)\ni \calZ_{\calU, h}\mapsto H_{\kappa_{\calU}}^{\tol, \leq h}$$ is a coherent sheaf on $\calZ$.  \\

We use the similar notation when considering the cohomology groups with compact supports. If $(\calU, h)$ is a slope datum for both $\calZ$ and $\calZ^c$, we have $$C_{c, \kappa_{\calU}}^{\tol, \leq h}\rightarrow C_{\kappa_{\calU}}^{\tol, \leq h}$$ due to the definition of the slope $\leq h$-decomposition and the Hecke-equivariance of the map $C_{c, \kappa_{\calU}}^{\tol}\rightarrow C_{\kappa_{\calU}}^{\tol}$. Hence we have a Hecke-equivariant map $$H_{c, \kappa_{\calU}}^{\tol, \leq h}\rightarrow H_{\kappa_{\calU}}^{\tol, \leq h}$$ which preserves the degrees. We then define $$H_{\Par, \kappa_{\calU}}^{\tol, \leq h}:=\image(H_{c, \kappa_{\calU}}^{\tol, \leq h}\rightarrow H_{\kappa_{\calU}}^{\tol, \leq h})$$ as well as the analogous notation for each degree.  

\begin{Proposition}\label{Proposition: parabolic cohomolgies gives coherent sheaves on the Fredholm hypersurface}
The assignment $$\Cov_{\sd}(\calZ)\ni \calZ_{\calU, h}\mapsto H_{\Par, \kappa_{\calU}}^{\tol, \leq h}$$ defines a coherent sheaf on $\calZ$, denoted by $\scrH_{\Par}^{\tol}$.
\end{Proposition}
\begin{proof} We need to show that for any slope data $(\calU, h)$ and $(\calV, h)$ for $\calZ$ with $\calV\subset \calU$ being a rational open subset, we have $H_{\Par, \kappa_{\calU}}^{\tol, \leq h}\otimes_{R_{\calU}}R_{\calV}=H_{\Par, \kappa_{\calV}}^{\tol, \leq h}$. This is the same to show the existence of the commutative diagram $$\begin{tikzcd}
H_{c, \kappa_{\calU}}^{\tol, \leq h}\otimes_{R_{\calU}}R_{\calV}\arrow[r]\arrow[d] & H_{\kappa_{\calU}}^{\tol, \leq h}\otimes_{R_{\calU}} R_{\calV}\arrow[d]\\
H_{c, \kappa_{\calV}}^{\tol, \leq h}\arrow[r] & H_{\kappa_{\calV}}^{\tol, \leq h}
\end{tikzcd}$$ whose vertical arrows are isomorphisms.\\

Observe that \[
    \scalemath{0.9}{D_{\kappa_{\calU}}^{\dagger}(\T_0, R_{\calU})\widehat{\otimes}_{R_{\calU}}R_{\calV}\simeq R_{\calU}\llbrack U_{\GSp_{2g}, 1}^{\opp}\rrbrack^{\wedge, \dagger}\widehat{\otimes}_{R_{\calU}}R_{\calV}\simeq R_{\calV}\llbrack U_{\GSp_{2g}, 1}^{\opp}\rrbrack^{\wedge, \dagger}\simeq D_{\kappa_{\calV}}^{\dagger}(\T_0, R_{\calV}),}
\]where the superscript ``$\bullet^{\wedge, \dagger}$'' means taking the completion with respect to the family of norms $(||\cdot||_r)_{r}$. Since both $C_{\kappa_{\calU}}^{\tol, \leq h}$ and $C_{c, \kappa_{\calU}}^{\tol, \leq h}$ are of finite presentation over $R_{\calU}$, taking cohomology commutes with flat base change, so \begin{align*}
    \scalemath{1}{H^t(X_{\Iw^+}(\C), D_{\kappa_{\calU}}^{\dagger})^{\leq h}\otimes_{R_{\calU}} R_{\calV}} & \simeq \scalemath{1}{H^t(X_{\Iw^+}(\C), D_{\kappa_{\calU}}^{\dagger}(\T_0, R_{\calU})\widehat{\otimes}_{R_{\calU}}R_{\calV})^{\leq h}}\\
    & \simeq \scalemath{1}{H_c^t(X_{\Iw^+}(\C), D_{\kappa_{\calV}}^{\dagger})^{\leq h}}\\
    \scalemath{1}{H_c^t(X_{\Iw^+}(\C), D_{\kappa_{\calU}}^{\dagger})^{\leq h}\otimes_{R_{\calU}} R_{\calV}} & \simeq \scalemath{1}{H_c^t(X_{\Iw^+}(\C), D_{\kappa_{\calU}}^{\dagger}(\T_0, R_{\calU})\widehat{\otimes}_{R_{\calU}}R_{\calV})^{\leq h}}\\
    & \simeq \scalemath{1}{H^t(X_{\Iw^+}(\C), D_{\kappa_{\calV}}^{\dagger})^{\leq h},}
\end{align*} where the first isomorphisms for both rows follow from that we are considering the finite slope parts. 
\end{proof}

For any slope datum $(\calU, h)$, the action of $\bbT$ on $H_{\Par, \kappa_{\calU}}^{\tol, \leq h}$ yields a morphism of commutative algebras $\bbT\rightarrow \End_{\scrO_{\calW}(\calU)}(H_{\Par, \kappa_{\calU}}^{\tol, \leq h})$ whose image is denoted by $\bbT_{\Par}^{ \calU, h}$, which is a finite algebra over $R_{\calU}$ since $H_{\Par, \kappa_{\calU}}^{\tol, \leq h}$ is finitely generated. Since $\scrH_{\Par}^{\tol}$ is a coherent sheaf, the assignment $$\scrT_{\!\!\Par}:\Cov_{\sd}(\calZ)\ni \calZ_{\calU, h}\mapsto \bbT_{\Par}^{\calU, h}$$ is a coherent sheaf of $\scrO_{\calZ}$-algebras. Then the cuspidal eigenvariety $\calE_{0}$ is defined to be $$\calE_0:=\Spa_{\calZ}(\scrT_{\!\!\Par}, \scrT_{\!\!\Par}^{\circ}),$$ where the sheaf of integral elements $\scrT_{\!\!\Par}^{\circ}$ is determined by \cite[Lemma A.3]{Johansson-Newton}. We name the structure morphisms $$\begin{tikzcd}
\calE_0\arrow[rr, "\pi^{\calE_0}_{\calZ}"]\arrow[rrrr, bend right = 20, "\wt"'] && \calZ\arrow[rr, "\wt_{\calZ}"] && \calW
\end{tikzcd}.$$ We further let $\calZ_{\Par}:= \image \pi_{\calZ}^{\calE_0}$ to be the Fredholm hypersurface corresponding to $\calE_0$.

\begin{Remark}\label{Remark: embedding the cuspidal eigenvariety into the whole eigenvariety}
\normalfont By applying the eigenvariety construction in \cite{Johansson-Newton} directly to $\GSp_{2g}$, one obtains the eigenvariety $\calE$, parametrising the Hecke eigenvectors of finite slope cohomology groups $H^{\bullet}(X_{\Iw^+}(\C), D_{\kappa}^{\dagger})^{\leq h}$. Recall that we have a Hecke-equivariant diagram $$\begin{tikzcd}
H_c^t(X_{\Iw^+}(\C), D_{\kappa}^\dagger)\arrow[r] & H^t(X_{\Iw^+}(\C), D_{\kappa}^\dagger)\arrow[r, "\pi"] & H_{\partial}^t(X_{\Iw}(\C), D_{\kappa}^{\dagger})\\ & H_{\Par}^t(X_{\Iw^+}(\C), D_{\kappa}^{\dagger})\arrow[u, hook]
\end{tikzcd}$$ for each $t$ such that $$H_{\Par}^t(X_{\Iw^+}(\C), D_{\kappa}^{\dagger})=\image\left(H_c^t(X_{\Iw^+}(\C), D_{\kappa}^r)\rightarrow H^t(X_{\Iw^+}(\C), D_{\kappa}^{\dagger})\right)=\ker \pi,$$ where the last equation is given by the exactness of the long exact sequence of cohomology groups. Let $\bbT^{\calU, h}$ be the image of $\bbT$ in $\End_{\scrO_{\calW}(\calU)}(H_{\kappa_{\calU}}^{\tol, \leq h})$, then there is a surjection $\bbT^{\calU, h}\twoheadrightarrow \bbT_{\Par}^{\calU, h}$ given by the restriction. Hence, one sees that there is a closed immersion $\calE_0\hookrightarrow \calE$ of adic spaces over $\calW$ and one views $\calE_0$ as the ``cuspidal part'' of $\calE$ since $\pi|_{H_{\Par}^t(X_{\Iw^+}(\C), D_{\kappa}^{\dagger})}$ is the zero map.
\end{Remark}

\begin{Remark}\label{Remark: Hansen's eigenvariety}
\normalfont As pointed out in \cite[Remark 4.1.9]{Johansson-Newton}, the eigenvariety constructed in \cite{Hansen-PhD} is the open locus of ${p\neq 0}$ inside $\calE$. Consequently, the cuspidal part of the eigenvariety in \textit{op. cit.} is the open locus of ${p\neq 0}$ inside our cuspidal eigenvariety $\calE_0$.
\end{Remark}

\begin{Corollary}\label{Corollary: pairing on the eigenvariety}
The pairing in Proposition \ref{Proposition: well-defined pairing on the parabolic cohomology} induce pairings $$\bls\cdot, \cdot\brs: \scrH_{\Par}^{\tol}\times \scrH_{\Par}^{\tol}\rightarrow \scrO_{\calZ}\text{ and }\bls \cdot, \cdot\brs: \pi^{\calE_0, *}_{\calZ}\scrH_{\Par}^{\tol}\times \pi^{\calE_0, *}_{\calZ}\scrH_{\Par}^{\tol}\rightarrow \scrO_{\calE_0}$$ of coherent sheaves on $\calZ$ and $\calE_0$ respectively. Moreover, the first pairing is $\bbT$-equivariant. 
\end{Corollary}
\begin{proof}
First of all, we claim that the pairing $\bls\cdot, \cdot\brs_{\kappa}^{\circ}$ is $\bfu_{p, i}$-equivariant for any $i=0, 1, ..., g-1$ and for any $p$-adic weight $\kappa: T_{\GL_g, 0}\rightarrow R$. Take any $\mu_1, \mu_2\in D_{\kappa}^{\dagger}$, we have
\begin{align*}
    \bls \bfu_{p, i}\cdot & \mu_1, \mu_2\brs_{\kappa}^{\circ} \\
    & = \scalemath{1}{\int_{\T_{00}^2} e_{\kappa}^{\hst}\left(\begin{pmatrix}\trans\bfgamma_2 & \trans\bfupsilon_2\end{pmatrix} \begin{pmatrix}\one_g\\ &  p^{-1}\one_g\end{pmatrix}\begin{pmatrix}\bfgamma_1\\ \bfupsilon_1\end{pmatrix}\right)\quad d\bfu_{p,i}\cdot\mu_1(\bfgamma_1, \bfupsilon_1)d\mu_2(\bfgamma_2, \bfupsilon_2)}\\
    & = \scalemath{1}{\int_{\T_{00}}e_{\kappa}^{\hst}\left(\trans\bfgamma_2\bfgamma_1+\trans\bfupsilon_2\bfupsilon_1/p\right)\quad d\bfu_{p, i}\cdot\mu_1(\bfgamma_1, \bfupsilon_1)d\mu_2(\bfgamma_2, \bfupsilon_2)}\\
    & = \scalemath{1}{\int_{\T_{00}^2}e_{\kappa}^{\hst}\left(\trans\bfgamma_2 (\bfu_{p,i}^{\square}\bfgamma_1\bfu_{p,i}^{\square, -1}) + \trans\bfupsilon_2(\bfu_{p,i}^{\blacksquare} \bfupsilon_1\bfu_{p, i}^{\square, -1})/p\right)\quad d\mu_1(\bfgamma_1, \bfupsilon_1)d\mu_2(\bfgamma_2, \bfupsilon_2)}\\
    & = \scalemath{1}{\int_{\T_{00}^2} e_{\kappa}^{\hst}\left((\trans\bfgamma_2\bfu_{p,i}^{\square}\bfgamma_1+\trans\bfupsilon_2\bfu_{p,i}^{\blacksquare}\bfupsilon_1/p)\bfu_{p, i}^{\square, -1}\right)\quad d\mu_1(\bfgamma_1, \bfupsilon_2)d\mu_2(\bfgamma_2, \bfupsilon_2)}\\
    & = \scalemath{1}{\int_{\T_{00}}e_{\kappa}^{\hst} \left(\bfu_{p,i}^{\square, -1}\trans\bfgamma_2\bfu_{p, i}^{\square}\bfgamma_1 + \bfu_{p, i}^{\square, -1}\trans\bfupsilon_{2}\bfu_{p,i}^{\blacksquare}\bfupsilon_1/p\right)\quad d\mu_1(\bfgamma_1, \bfupsilon_1)d\mu_2(\bfgamma_2, \bfupsilon_2)}\\
    & = \scalemath{1}{\int_{\T_{00}}e_{\kappa}^{\hst}\left(\trans(\bfu_{p,i}^{\square}\bfgamma_2\bfu_{p,i}^{\square, -1})\bfgamma_1 + \trans(\bfu_{p,i}^{\blacksquare}\bfupsilon_2\bfu_{p, i}^{\square, -1})\bfupsilon_1/p\right) \quad d\mu_1(\bfgamma_1, \bfupsilon_1)d\mu_2(\bfgamma_2, \bfupsilon_2)}\\
    & = \bls \mu_1, \bfu_{p,i}\cdot \mu_2\brs_{\kappa}^{\circ},
\end{align*} where the antepenultimate equation follows from the nature of determinants (which was used in the definition of $e_{\kappa}^{\hst}$). \\

This claim then implies that we have a $U_{p, i}$-equivariant pairing $$\bls\cdot, \cdot\brs_{\kappa_{\calU}}: H_{\Par, \kappa_{\calU}}^{\tol, \leq h}\times H_{\Par, \kappa_{\calU}}^{\tol, \leq h}\rightarrow R_{\calU}.$$ Thus, by gluing, one obtains the first desired pairing. It is furthermore $\bbT^p$-equivariant since the Hecke operators outside $p$ acts on the ananlytic distributions trivially. The second one follows immediately.
\end{proof}

%% file: 4-Ramification.tex
\section{The ramification locus of the cuspidal eigenvariety} In this section, we apply our pairing to the study of the ramification locus of the cuspidal eigenvariety for $\GSp_{2g}$. We will first set up a formalism by following the strategy in \cite{Bellaiche-eigenbook}. Then, the main results of this paper are proven in Theorem \ref{Theorem: ramification locus}, Theorem \ref{Theorem: vanishing order and the ramification index} and Corollary \ref{Corollary: main result of the paper}.

\subsection{Some commutative algebra} In this subsection, we recollect some ingredients of commutative algebra from \cite{Bellaiche-eigenbook}. \\

Let $A$ be a noetherian domain and $B$ be a finite flat $A$-algebra. Consider $$\mult: B\otimes_A B\rightarrow B, \quad b\otimes b'\mapsto bb'$$ and write $\multideal = \ker(\mult)$. Let $$(B\otimes_A B)[\multideal] : = \{x\in B\otimes_A B: y\cdot x= 0 \,\,\forall y\in \multideal\},$$ then the \textit{\textbf{Noether's different}} of $B$ over $A$ is defined to be the ideal 
\[
    \frakd(B/A) : = \image\left((B\otimes_A B)[\multideal] \xrightarrow{\mult} B\right)
\] in $B$.

\begin{Theorem}[Auslander--Buchsbaum]
A prime ideal $\frakP$ of $B$ is ramified over $A$ if and only if $\frakd(B/A)\subset \frakP$. Equivalently, $\Spec B/\frakd(B/A)$ is the ramification locus of $\Spec B$ over $\Spec A$.
\end{Theorem}
\begin{proof}
See \cite[Theorem 2.7]{Auslander-Buchsbaum}.
\end{proof}

Suppose $M, N$ are two $B$-modules which are finite flat over $A$ and assume we are in the following situation:\begin{enumerate}
    \item[$\bullet$] There exists an $A$-linear pairing $$\beta: M\times N \rightarrow A$$ such that $\beta$ is $B$-equivariant. 
    
    \item[$\bullet$] We have isomorphisms $M\simeq N\simeq B^{\vee}:= \Hom_{A}(B, A)$ of $B$-modules.  
\end{enumerate}

\begin{Lemma}[\text{\cite[Proposition VIII.1.11]{Bellaiche-eigenbook}}]\label{Lemma: commutaitve algebra lemma for the principal-ness of the adjoint L-ideal}
Denote by $\beta_{B}$ the base change of $\beta$ to $B$ on $M\otimes_A B\times N\otimes_A B$. Let \begin{align*}
    (M\otimes_A B)[\multideal] & = \{x \in M\otimes_A B: y\cdot x=0\,\,\forall y\in \multideal\}\\
    (N\otimes_A B)[\multideal] & = \{ x \in N\otimes_A B: y\cdot x=0\,\,\forall y\in \multideal\}.
\end{align*} Then the ideal $$\frakL_{\beta}:=\image\left(\beta_B: (M\otimes_A B)[\multideal]\times (N\otimes_A B)[\multideal]\rightarrow B\right)$$ is a principal ideal in $B$.
\end{Lemma}
\begin{proof}
We claim first that for any $B$-module $M$ which is finite flat over $A$, we have an isomorphism $M^\vee\otimes_A B[\multideal]\simeq \Hom_B(M, B)$, where $M^{\vee}=\Hom_A(M, A)$. Notice that $M^{\vee}$ also admits a $B$-module structure by $b\psi:m\mapsto \psi(bm)$ for all $b\in B$, $\psi\in M^\vee$ and $m\in M$. We have a natural isomorphism $$M^\vee\otimes_A B=\Hom_A(M, A)\otimes_A B\rightarrow \Hom_A(M, B), \quad \psi\otimes b\mapsto (m\mapsto \psi(m)b).$$ Since $\multideal=\sum_{b\in B}(b\otimes 1-1\otimes b)B\otimes_A B$, thus \begin{align*}
    \psi\otimes b\in M^\vee\otimes_A B[\multideal] & \Leftrightarrow (b'\otimes 1-1\otimes b')\psi\otimes b=0\,\,\forall b'\in B\\
    & \Leftrightarrow b'\psi\otimes b=\psi\otimes bb'\,\,\forall b'\in B\\
    & \Leftrightarrow \psi(b'm)b=\psi(m)bb'\,\,\forall b'\in B, m\in M\\
    & \Leftrightarrow (m\mapsto \psi(m)b)\in \Hom_B(M, B).
\end{align*}

Apply the claim in our situation, we have isomorphisms of $B$-modules $$(M\otimes_A B)[\multideal]\simeq (B^{\vee}\otimes_A B)[\multideal]\simeq \Hom_B(B, B)\simeq B$$ and same for $(N\otimes_A B)[\multideal]$. Hence, let $\widetilde{m}$ and $\widetilde{n}$ be generators of $(M\otimes_A B)[\multideal]$ and $(N\otimes_A B)[\multideal]$ respectively as $B$-modules. Then $\frakL_{\beta}=\beta_B(\widetilde{m}, \widetilde{n})B$.
\end{proof}

\begin{Proposition}[\text{\cite[Corollary VIII.1.13]{Bellaiche-eigenbook}}]\label{Proposition: relation of L-ideal and Noether's different}
Suppose $B$ is Gorenstein over $A$, i.e., $B^{\vee}$ is flat of constant rank $1$ over $B$, and $M, N$ are $B$-modules which are finite flat over $A$ and flat of rank 1 over $B$. Assume there is an $A$-linear pairing $\beta: M\times N\rightarrow A$ which is $B$-equivariant. We retain the notation $\beta_B$ and $\frakL_{\beta}$ as in Lemma \ref{Lemma: commutaitve algebra lemma for the principal-ness of the adjoint L-ideal}. Then \begin{enumerate}
    \item Both ideals $\frakd(B/A)$ and $\frakL_{\beta}$ are locally principal. Moreover, there exists $b_0\in B$ such that $\frakL_{\beta}=b_0\frakd(B/A)$.
    \item We have $\frakL_{\beta}=\frakd(B/A)$ if and only if $\beta$ is non-degenerate. 
\end{enumerate}
\end{Proposition}
\begin{proof}
We are in a special case of Lemma \ref{Lemma: commutaitve algebra lemma for the principal-ness of the adjoint L-ideal} that we can identify (locally) $M\simeq N\simeq B^{\vee}\simeq B$ and hence we know $\frakL_{\beta}$ is principal. Moreover, the identification $B\otimes_AB[\multideal]\simeq \Hom_B(B, B)\simeq B$ implies that $\frakd(B/A)$ is also principal. \\

Observe that we can identify $\beta:M\times N\rightarrow A$ as a linear morphism $B^\vee\otimes_A B^\vee\rightarrow A$. Hence by duality, we identify $\beta$ with an element $b\in B\otimes_A B$. We claim that $\frakL_{\beta}=\mult(b)B$. As we are working locally, we assume $b_1, ..., b_n$ is a basis of $B$ over $A$, then $b_1^{\vee}, ..., b_n^{\vee}$ is a basis of $B^{\vee}$ over $A$. Observe that $b^{\square}:=\sum_{i}b_i^{\vee}\otimes b_i$ is a generator of $B^{\vee}\otimes_A B[\multideal]\simeq \Hom_{B}(B, B)$ as it maps to the identity in $\Hom_B(B, B)$. Hence by definition $$\frakL_{\beta}=\beta_B(b^{\square}, b^{\square})B=\left(\sum_{i,j}\beta(b_i^{\vee}, b_j^{\vee})b_ib_j\right)B.$$ On the other hand, by the above construction, we see that $b=\sum_{i,j}\beta(b_i^{\vee}, b_j^{\vee})b_i\otimes b_j$ with $\mult(b)=\sum_{i,j}\beta(b_i^{\vee}, b_j^{\vee})b_ib_j$.\\

Let $\widetilde{b}^{\square}=\sum_{i}b_i\otimes b_i$, then it is a generator of $B\otimes_A B[\multideal]\simeq B$. Thus, there exists $b_0\in B$ such that $b_0\widetilde{b}^{\square}=b$. We conclude that $$\frakL_{\beta}=\mult(b)B= \mult(b_0\widetilde{b}^{\square})B=b_0\mult(\widetilde{b}^{\square})B=b_0\frakd(B/A).$$ Finally, we have \begin{align*}
    \frakL_{\beta}=\frakd(B/A) & \Leftrightarrow b_0\in B^\times\\ 
    & \Leftrightarrow \beta(b_i^{\vee}, b_j^{\vee})=\left\{\begin{array}{ll}
        b_0\in B^\times & i=j \\
        0 & i\neq j
    \end{array}\right.\\
    & \Leftrightarrow \beta\text{ is non-degenerate}.
\end{align*}
\end{proof}

\subsection{The ramification locus of the cuspidal eigenvariety} Recall the weight map $\wt: \calE_0\rightarrow \calW$ and $\pi^{\calE_0}_{\calZ}: \calE_0\rightarrow \calZ$. For each slope datum $(\calU, h)$, let $\calE_0^{\calU, h}:=(\pi^{\calE_0}_{\calZ})^{-1}(\calZ_{\calU, h})$. We adapt the definitions of ``clean neighbourhoods'' and ``good points'' in \cite{Bellaiche-eigenbook} in our situation:

\begin{Definition}\label{Definition: clean neighbourhood and good points}
\begin{enumerate}
    \item Let $\bfx\in \calE_0$ and $\calV=\Spa(R_{\calV}, R_{\calV}^+)$ be an open affinoid neighbourhood of $\bfx$. We say $\calV$ is a \textbf{clean neighbourhood} of $\bfx$ if it satisfies the following properties:\begin{enumerate}
        \item[$\bullet$] $\wt(\calV)=\calY=\Spa(R_{\calY}, R_{\calY}^+)\subset \calW$ is an open affinoid subset of $\calW$ and there exists a slope datum $(\calU, h)$ for $\calZ$ such that $\calV$ is the connected component of $\bfx$ in $\calE_0^{\calU, h}$;
        \item[$\bullet$] $\bfx$ is the only point of $\calV$ sitting above $\wt(\bfx)$;
        \item[$\bullet$] the map $\wt: \calV\rightarrow \calY$ is flat and is moreover \'{e}tale except perhaps at $\bfx$.
    \end{enumerate} In this case, there exists an idempotent $\eta=\eta_{\calV}\in \bbT_{\Par}^{\calU, h}$ such that $\calV$ is defined by the equation $\eta=1$ and the module $\eta H_{\Par, \kappa_{\calU}}^{\tol, \leq h}$ is a direct summand of $H_{\Par, \kappa_{\calU}}^{\tol, \leq h}$.
    \item A point $\bfx\in \calE_0$ is said to be a \textbf{good point} if it admits a sufficiently small clean neighbourhood $\calV$ with $\wt(\calV)=\calY$ such that the modules $\eta_{\calV}H_{\Par, \kappa_{\calU}}^{\tol, \leq h}$ and $(\eta_{\calV} H_{\Par, \kappa_{\calU}}^{\tol, \leq h})^{\vee}$ are free of rank one over $R_{\calV}$, where the dual is taken to be an $R_{\calY}$-dual.
\end{enumerate}
\end{Definition}

\begin{Remark}\label{Remark: don't know if the eigenvariety is flat over the weight space}
\normalfont We remark the following: \begin{enumerate}
    \item In the $\GL_2$ case, the eigencurve is finite flat over the weight space (\cite[\S VI.1.4]{Bellaiche-eigenbook}) and so the author of \textit{op. cit.} can consequently deduce that the collection of clean neighbourhoods of points on the eigencurve gives a open cover of the eigencurve. In our case, the Fredholm hypersurface $\calZ$ is finite flat over $\calW$ by \cite[Theorem B.1]{AIP-2018}. However, we don't know if $\calE_0$ is flat over $\calZ$. Therefore, instead of considering $\calE_0$, we consider $\calE_0^{\fl}\subset \calE_0$ the flat locus over $\calW$, which is open over $\calW$, and let $\Cov_{\cl}(\calE_0^{\fl})$ be the open cover of clean neighbourhoods. 
    \item In the definition of good points, we see immediately that $R_{\calV}$ is Gorenstein over $R_{\calY}$.
\end{enumerate}
\end{Remark}

Following \cite[\S VIII. 4]{Bellaiche-eigenbook}, we study the \textit{adjoint $L$-ideal} and define the \textit{$p$-adic adjoint $L$-function} here. Let $\bfx\in \calE_0^{\fl}$ and $\calV$ be a clean neighbourhood of $\bfx$ with weight $\wt(\calV)=\calY$. There is a natural multiplication map $$\mult: R_{\calV}\widehat{\otimes}_{R_{\calY}}R_{\calV}\rightarrow R_{\calV}, \quad b\otimes b'\mapsto bb'.$$ Let $\multideal:=\ker\mult$ and define $$M\widehat{\otimes}_{R_{\calY}}R_{\calV}[\multideal]:=\{m\in M\widehat{\otimes}_{R_{\calY}}R_{\calV}: \multideal\cdot m=0\}$$ for any Banach $R_{\calV}$-module $M$.

\begin{Definition}\label{Definition: the adjoint L-ideal}
Keep the above notation. The \textbf{adjoint $L$-ideal} of $\calV$ is defined to be \[
    \scalemath{0.9}{\scrL^{\adj}(\calV):=\image\left(\bls\cdot ,\cdot\brs_{\kappa_{\calU}}: \eta_{\calV}H_{\Par, \kappa_{\calU}}^{\tol, \leq h}\widehat{\otimes}_{R_{\calY}}R_{\calV}[\multideal]\times \eta_{\calV}H_{\Par, \kappa_{\calU}}^{\tol, \leq h}\widehat{\otimes}_{R_{\calY}}R_{\calV}[\multideal]\rightarrow R_{\calV}\right).}
\]
\end{Definition}

\begin{Remark}\label{Remark: adjoint L-ideal as a sheaf}
\normalfont Since the clean neighbourhoods cover $\calE_0^{\fl}$, the collection $\{\scrL^{\adj}(\calV): \calV\in \Cov_{\cl}(\calE_0^{\fl})\}$ glues to a coherent sheaf $\scrL^{\adj}$ on $\calE_0^{\fl}$.
\end{Remark}

\begin{Proposition}\label{Proposition: adjoint L-dieal at a good point is principal}
Let $\bfx\in \calE_0^{\fl}$ be a good point. Then there exists a sufficiently small clean neighbourhood $\calV$ of $\bfx$ with $\wt(\calV)=\calY$ such that $\scrL^{\adj}(\calV)$ is a principal ideal in $R_{\calV}$. 
\end{Proposition}
\begin{proof}
The assertion follows from Lemma \ref{Lemma: commutaitve algebra lemma for the principal-ness of the adjoint L-ideal}.
\end{proof}

\begin{Definition}\label{Definition: the adjoint p-adic L-function}
Let $\bfx\in \calE_0^{\fl}$ be a good point and $\calV$ be a sufficiently small clean neighbourhood such that $\scrL^{\adj}(\calV)$ is principal. We define the \textbf{adjoint $p$-adic $L$-function} on $\calV$ to be $L_{\calV}^{\adj}\in R_{\calV}$ such that $L_{\calV}^{\adj}$ generates $\scrL^{\adj}(\calV)$. The value of $L_{\calV}^{\adj}$ at $\bfx$ is denoted by $L^{\adj}(\bfx)$ as it doesn't depend on the clean neighbourhood.
\end{Definition}

\begin{Remark}\label{Remark: adjoint p-adic L-function}
\normalfont We point out that the adjoint $p$-adic $L$-function $L_{\calV}^{\adj}$ is defined up to a unit in $R_{\calV}$. In the case of $\GL_2$, the name ``adjoint $p$-adic $L$-function'' is justified in \cite[Proposition 3.9.2]{Kim} and \cite[\S VIII.5]{Bellaiche-eigenbook}. However, the justification of the name is unknown to us in our situation as discussed in the introduction.  
\end{Remark}

Let $\bfx\in \calE_0^{\fl}$ be a good point and let $\calV$ be a sufficiently small clean neighbourhood of $\bfx$ such that $L_{\calV}^{\adj}$ is defined. Let $(\calU, h)$ be the slope datum that defines $\calV$ and let $\wt(\calV)=\calY$. Corollary \ref{Corollary: pairing on the eigenvariety} yields an $R_{\calV}$-equivariant pairing $$\bls\cdot, \cdot\brs_{\kappa_{\calU}}: \eta_{\calV}H_{\Par, \kappa_{\calU}}^{\tol, \leq h}\times \eta_{\calV}H_{\Par, \kappa_{\calU}}^{\tol, \leq h}\rightarrow R_{\calY}.$$ Together with the definition of good points, we are in the situation of Proposition \ref{Proposition: relation of L-ideal and Noether's different}.

\begin{Theorem}\label{Theorem: ramification locus}
Let $\bfx\in \calE_0^{\fl}$ be a good point and let $\kappa=\wt(\bfx)$. Suppose the pairing $$\bls\cdot, \cdot\brs_{\kappa_{\calU}}: \eta_{\calV}H_{\Par, \kappa_{\calU}}^{\tol, \leq h}\times \eta_{\calV}H_{\Par, \kappa_{\calU}}^{\tol, \leq h}\rightarrow R_{\calY}$$ is non-degenerate at $\wt(\bfx)$, then $$L^{\adj}(\bfx)=0\text{ if and only if }\wt \text{ is ramified at }\bfx.$$
\end{Theorem}
\begin{proof}
Let $\calV$ be a sufficiently small clean neighbourhood of $\bfx$ which is defined by the slope datum $(\calU, h)$ and $\wt(\calV)=\calY$. Since the pairing $$\bls\cdot, \cdot\brs_{\kappa_{\calU}}: \eta_{\calV}H_{\Par, \kappa_{\calU}}^{\tol, \leq h}\times \eta_{\calV}H_{\Par, \kappa_{\calU}}^{\tol, \leq h}\rightarrow R_{\calY}$$ is assumed to be non-degenerate, then by Proposition \ref{Proposition: relation of L-ideal and Noether's different}, $\scrL^{\adj}(\calV)=\frakd(R_{\calV}/R_{\calY})$. Thus, \begin{align*}
    L^{\adj}(\bfx)=0 & \Leftrightarrow L^{\adj}\in \supp \bfx \Leftrightarrow \frakd(R_{\calV}/R_{\calY})\subset \supp\bfx\Leftrightarrow \wt\text{ is ramified at }\bfx,
\end{align*} where the last equivalence is due to Auslander--Buchsbaum's theorem.
\end{proof}

\begin{Theorem}\label{Theorem: vanishing order and the ramification index}
Let $\bfx\in \calE_0^{\fl}$ be a good and smooth point and let $\kappa=\wt(\bfx)$. We further assume $\bfx$ lives in the open locus of $\calE_0^{\fl}$ where $p\neq 0$. Assume again that the pairing $$\bls\cdot, \cdot\brs_{\kappa_{\calU}}: \eta_{\calV}H_{\Par, \kappa_{\calU}}^{\tol, \leq h}\times \eta_{\calV}H_{\Par, \kappa_{\calU}}^{\tol, \leq h}\rightarrow R_{\calY}$$ is non-degenerate at $\wt(\bfx)$. Let $R_{\wt(\bfx)}$ and $R_{\bfx}$ be the local rings at $\wt(\bfx)$ and $\bfx$ respectively and denote by $\frakm_{\wt(\bfx)}$, $\frakm_{\bfx}$ their maximal ideals respectively. Let $\Fitt_{\bfx}$ be the 0-th Fitting ideal of $\Omega_{R_{\bfx}/R_{\wt(\bfx)}}^1$ and define \[
 e(\bfx) : = \max\{e\in\Z_{\geq 0}: \Fitt_{\bfx}\subset \frakm_{\bfx}^e\}.
\] Then, we have \[
    \ord_{\bfx}L^{\adj} = e(\bfx).
\]
%Denote by $e(\bfx)$ the ramification index of $\wt$ at $\bfx$, then it is related with the vanishing order of $L^{\adj}(\bfx)$ by the formula $$\ord_{\bfx}L^{\adj}=e(\bfx)-1.$$
\end{Theorem}
\begin{proof}
By \cite[Theorem VIII. 1.4]{Bellaiche-eigenbook}, we have  $\frakd(R_{\bfx}/R_{\kappa})=\Fitt_{\bfx}$ (since $\bfx$ is a smooth point) and so $$e(\bfx)=\max\{e\in \Z_{\geq 0}: \frakd(R_{\bfx}/R_{\kappa})\subset \frakm_{\bfx}^e\}.$$ On the other hand, $$\ord_{\bfx}L^{\adj}:=\max\{e\in \Z_{\geq 0}: L^{\adj}(\bfx)\in \frakm_{\bfx}^{e}\}.$$ In our situation, we see that $$\frakm_{\bfx}^{e(\bfx)}\supset \frakd(R_{\bfx}/R_{\kappa})=L^{\adj}(\bfx)R_{\bfx}\subset \frakm_{\bfx}^{\ord_{\bfx}L^{\adj}}.$$ As the inclusions on both sides satisfy the same condition, the exponents coincide.
\end{proof}

\begin{Remark}\label{Remark: comparing our work to Bellaiche}
\normalfont We remark that the above two theorems have their roots in the $\GL_2$ case. Theorem \ref{Theorem: ramification locus} is an analogue of \cite[Theorem VIII.4.7]{Bellaiche-eigenbook} while Theorem \ref{Theorem: vanishing order and the ramification index} is inspired by [\textit{op. cit.}, Theorem VIII.4.8(i)].
\end{Remark}

\subsection{Non-degeneracy of the pairing}
In the statements of Theorem \ref{Theorem: ramification locus} and Theorem \ref{Theorem: vanishing order and the ramification index}, we assumed that the pairing $\bls\cdot, \cdot\brs_{\wt(\bfx)}$ is non-degenerate at $\wt(\bfx)$. In this subsection, we justify that such an assumption is not vacuous. \\

Let $\kappa: T_{\GL_g,0}\rightarrow R$ be any $p$-adic weight. Recall the pairing $\bls\cdot, \cdot\brs_{\kappa}$ on the parabolic cohomology groups is defined by the pairing 
\begin{align*}
    \bls\cdot, \cdot\brs_{\kappa}^{\circ}: &  D_{\kappa}^{\dagger}(\T_0, R)\times D_{\kappa}^{\dagger}(\T_0, R)\rightarrow R, \\
    & (\mu_1, \mu_2)\mapsto \int_{\T_{00}^2}e_{\kappa}^{\hst}\left(\begin{pmatrix}\trans\bfgamma_2 & \trans\bfupsilon\end{pmatrix}\begin{pmatrix}\one_g\\ & p^{-1}\one_g\end{pmatrix}\begin{pmatrix}\bfgamma_1\\ \bfupsilon_1\end{pmatrix}\right)\quad d\mu_1 d\mu_2.
\end{align*} When $\kappa = k\in \Z_{>0}^g$ is a dominant algebraic weight, recall that we also have algebraic representations $\V_{\GSp_{2g}, k}^{\alg}$ and $\V_{\GSp_{2g}, k}^{\alg, \vee}$ defined in \S \ref{subsection: pairing on the analytic distributions}. From now on, we abuse the notation, writing $\V_{\GSp_{2g}, k}^{\alg}$ and $\V_{\GSp_{2g}, k}^{\alg, \vee}$ for their $\Q_p$-realisation. That is, \begin{align*}
    \V_{\GSp_{2g}, k}^{\alg} & = \scalemath{1}{\left\{\phi: \GSp_{2g}(\Q_p)\rightarrow \Q_p: \begin{array}{l}
        \bullet \,\,\phi\text{ is a polynomial function }  \\
        \bullet\,\,\phi(\bfgamma\bfbeta) = k(\bfbeta)\phi(\bfgamma) \\
        \,\,\,\,\, \forall (\bfgamma, \bfbeta)\in \GSp_{2g}(\Q_p)\times B_{\GSp_{2g}}(\Q_p)  
    \end{array}\right\}}\\
    \V_{\GSp_{2g}, k}^{\alg, \vee} & = \Hom_{\Q_p}(\V_{\GSp_{2g}, k}^{\alg}, \Q_p).
\end{align*} There is an obvious injective morphism 
\[
    \V_{\GSp_{2g}, k}^{\alg}\rightarrow A_k^r(\T_0, \Q_p), \quad \phi\mapsto \left((\bfgamma, \bfupsilon)\mapsto k(\bfbeta)\phi\left(\begin{pmatrix}\bfgamma_0\\ \bfupsilon_0 & \oneanti_g \trans\bfgamma_0^{-1}\oneanti_g\end{pmatrix}\right)\right)
\] for any $r$, where $(\bfgamma, \bfupsilon) = (\bfgamma_0, \bfupsilon_0)\bfbeta$ with $\bfgamma_0\in U_{\GL_g, 1}^{\opp}$ and $\bfbeta\in B_{\GL_g,0}^+$. Therefore, there is a natural surjection $D_k^{\dagger}(\T_0, \Q_p)\rightarrow \V_{\GSp_{2g}, k}^{\alg, \vee}$, which is $\Iw_{\GSp_{2g}}^+$-equivariant. We then descend the pairing $\bls\cdot, \cdot\brs_k^{\circ}$ to $\V_{\GSp_{2g}, k}^{\alg, \vee}$ by the same formula
\begin{align*}
    \bls\cdot, \cdot\brs_k^{\circ}: & \V_{\GSp_{2g}, k}^{\alg, \vee} \times \V_{\GSp_{2g}, k}^{\alg, \vee}\rightarrow \Q_p, \\
    & (\mu_1, \mu_2)\mapsto \scalemath{1}{\int_{\bfgamma_1,\bfgamma_2\in U_{\GSp_{2g}, 1}^{\opp}} e_k^{\hst}\left(\trans\bfgamma_2\begin{pmatrix}\one_g\\ & p^{-1}\one_g\end{pmatrix}\bfgamma_1\right)\quad d\mu_1(\bfgamma_1)d\mu_2(\bfgamma_2).}
\end{align*}

\begin{Proposition}\label{Proposition: non-degeneracy on algebriac representation}
Let $k\in \Z_{>0}^g$ be a dominant weight. Then the pairing $\bls\cdot, \cdot\brs_{k}^{\circ}$ on $\V_{\GSp_{2g}, k}^{\alg, \vee}$ is non-degenerate.
\end{Proposition}
\begin{proof}
Recall the symplectic pairing $\bla\cdot, \cdot\bra_k$ on $\V_{\GSp_{2g}, k}^{\alg, \vee}$ from Remark \ref{Remark: symplecti pairing}
\[
    \bla \mu_1, \mu_2\bra_k = \int_{\bfgamma_1, \bfgamma_2\in \GSp_{2g}(\Q_p)}e_{k}^{\hst}\left(\trans\bfgamma_2\begin{pmatrix} & -\oneanti_g\\ \oneanti_g\end{pmatrix}\bfgamma_1\right)\quad d\mu_1(\bfgamma_1)d\mu_2(\bfgamma_2).
\] Since the symplectic pairing $\bla\cdot, \cdot\bra$ on $\V_{\Z}$ is non-degenerate, we know that $\bla\cdot, \cdot\bra_k$ is non-degenerate.\\

Define 
\begin{align*}
    \bla\cdot, \cdot\bra_k': & \V_{\GSp_{2g}, k}^{\alg, \vee}\times \V_{\GSp_{2g}, k}^{\alg, \vee}\rightarrow \Q_p, \\
    & (\mu_1, \mu_2)\mapsto \scalemath{1}{\int_{\bfgamma_1, \bfgamma_2\in U_{\GSp_{2g}}^{\opp}(\Q_p)}e_k^{\hst}\left(\trans\bfgamma_2\begin{pmatrix} & -\oneanti_g\\ \oneanti_g\end{pmatrix} \bfgamma_1\right)\quad d\mu_1(\bfgamma_1)d\mu_2(\gamma_2).}
\end{align*} Then $\bla\cdot, \cdot\bra_k'$ is a non-degenerate pairing. Indeed, we have \begin{align*}
    \bla \mu_1, \mu_2\bra_k &  = \scalemath{1}{\int_{\bfgamma_1, \bfgamma_2\in \GSp_{2g}(\Q_p)}e_{k}^{\hst}\left(\trans\bfgamma_2\begin{pmatrix} & -\oneanti_g\\ \oneanti_g\end{pmatrix}\bfgamma_1\right)\quad d\mu_1(\bfgamma_1)d\mu_2(\bfgamma_2)}\\
    & = \scalemath{1}{\int_{\bfgamma_1, \bfgamma_2\in \GSp_{2g}(\Q_p)}k(\beta_1)k(\beta_2)e_{k}^{\hst}\left(\trans\bfgamma_2'\begin{pmatrix} & -\oneanti_g\\ \oneanti_g\end{pmatrix}\bfgamma_1'\right)\quad d\mu_1(\bfgamma_1)d\mu_2(\bfgamma_2),}
\end{align*} where $\bfgamma_i=\bfgamma_i'\bfbeta$ with $\bfgamma_i'\in U_{\GSp_{2g}}^{\opp}(\Q_p)$ and $\bfbeta_i\in B_{\GSp_{2g}}(\Q_p)$ for $i=1, 2$. As $k$ is non-zero on $B_{\GSp_{2g}}(\Q_p)$, we see that $\bla\mu_1, \mu_2\bra_k = 0$ if and only if $\bla\mu_1, \mu_2\bra_k'=0$.\\

Now, let $\bls\cdot, \cdot\brs_k'$ be the pairing on $\V_{\GSp_{2g}, k}^{\alg, \vee}$ defined by 
\begin{align*}
    \bls \mu_1, \mu_2\brs_k' & := \bla \mu_1, \w_p\cdot\mu_2\bra_k' \\
    & = \int_{\bfgamma_1, \bfgamma_2\in U_{\GSp_{2g}}^{\opp}(\Q_p)}e_{k}^{\hst}\left(\trans\bfgamma_2\begin{pmatrix}\one_g \\ & p^{-1}\one_g\end{pmatrix}\bfgamma_1\right)\quad d\mu_1(\bfgamma_1)d\mu_2(\bfgamma_2).
\end{align*} Then, $\bls\cdot, \cdot\brs_k'$ is again a non-degenerate pairing since $\w_p\in \GSp_{2g}(\Q_p)$. Recall that $U_{\GSp_{2g}, 1}^{\opp}\simeq \Z_p^{d_0}$, for some $d_0\in \Z_{>0}$, as $p$-adic manifolds, thus $U_{\GSp_{2g}}^{\opp}(\Q_p)\simeq \Q_p^{d_0}$. However, $\V_{\GSp_{2g}, k}^{\alg, \vee}$ is defined algebraically and $\Z_p^{d_0}\subset \Q_p^{d_0}$ is Zariski dense, thus the non-degeneracy of $\bls\cdot, \cdot\brs_k'$ implies the non-degeneracy of $\bls\cdot, \cdot\brs_{k}^{\circ}$.
\end{proof}

\begin{Theorem}[Control theorem]\label{Theorem: control theorem}
For $g\in \Z_{>0}$, let $k=(k_1,..., k_g)\in \Z_{>0}^g$ be a dominant algebraic weight. Then, there exists $h_k\in \R_{>0}$ (depending on $k$) such that for any $\Q_{>0}\ni h< h_k$, we have a canonical isomorphism $$H^t_{\Par}(X_{\Iw^+}(\C), D_{k}^{\dagger})^{\leq h}\simeq H_{\Par}^t(X_{\Iw^+}(\C), \V_{\GSp_{2g}, k}^{\alg, \vee})^{\leq h}.$$ (see also \cite[Theorem 6.4.1]{Ash-Stevens})
\end{Theorem}
\begin{proof}
Let $\K:=\ker(D_k^{\dagger}(\T_0, \Q_p)\rightarrow \V_{\GSp_{2g}, k}^{\alg, \vee})$ and so we have an exact sequence 
\[
0\rightarrow C^{\bullet}(\Iw_{\GSp_{2g}}^+, \K)\rightarrow C^{\bullet}(\Iw_{\GSp_{2g}}^+, D_{k}^{\dagger})\rightarrow C^{\bullet}(\Iw_{\GSp_{2g}}^+, \V_{\GSp_{2g}, k}^{\alg, \vee})\rightarrow 0.
\] Define the $[\Iw_{\GSp_{2g}}^+ \bfu_{p,i}\Iw_{\GSp_{2g}}^+]$-action on $\V_{\GSp_{2g}, k}^{\alg, \vee}$ as the same formula on $D_{k}^{\dagger}(\T_0, \Q_p)$ but with $\bfu_{p,i}$ acting on $\V_{\GSp_{2g}, k}^{\alg, \vee}$ via the conjugation \[
    \bfu_{p,i} \cdot \bfgamma = \bfu_{p,i}\bfgamma \bfu_{p,i}^{-1}
\] on $\GSp_{2g}(\Q_p)$. We then see that the map \[ C^{\bullet}(\Iw_{\GSp_{2g}}^+, D_{k}^{\dagger})\rightarrow C^{\bullet}(\Iw_{\GSp_{2g}}^+, \V_{\GSp_{2g}, k}^{\alg, \vee}) \] is Hecke equivariant and so $C^{\bullet}(\Iw_{\GSp_{2g}}^+, \K)$ is Hecke stable. Denote by $C_{\K}^{\tol}$ and $C_{k, \alg}^{\tol}$ the total cochain complexes of $C^{\bullet}(\Iw_{\GSp_{2g}}^+, \K)$ and $C^{\bullet}(\Iw_{\GSp_{2g}}^+, \V_{\GSp_{2g}, k}^{\alg, \vee})$ respectively. Additionally, let $F_k^{\dagger}$ and $F_k^{\alg}$ be the Fredholm determinant of $U_p$ acting on $C_k^{\tol}$ and $C_{k,\alg}^{\tol}$ respectively.  We define \begin{align*}
    h_{\K} := &\, \scalemath{1}{\sup\left\{h\in \Q_{>0}: ||U_p||_{\K}:=\sup\left\{ \frac{||U_p\cdot \sigma||_{\dagger}}{||\sigma||_{\dagger}}: \,\,\forall \sigma\in  C_{\K}^{\tol}\right\}\leq p^{-h}\right\},}\\
    &\,\, \scalemath{1}{\text{where $||\cdot||_{\dagger}$ is the norm on $C_{k}^{\tol}$}}\\
    h_{\dagger} := &\, \scalemath{1}{\sup\{h\in \Q_{>0}: F_k^{\dagger}=Q^{\dagger}S^{\dagger}\text{ satisfying conditions in Proposition \ref{Proposition: properties of the slope decomposition of the Fredholm determinant} w.r.t. $h$}\}}\\
    h_{\alg} := &\, \scalemath{1}{\sup\{h\in \Q_{>0}: F_k^{\alg}=Q^{\alg}S^{\alg}\text{ satisfying conditions in Proposition \ref{Proposition: properties of the slope decomposition of the Fredholm determinant} w.r.t. $h$}\}}\\
    h_k := &\, \scalemath{1}{\min\{h_{\K}, h_{\dagger}, h_{\alg}\}}.
\end{align*}

Now, we claim the following: Fix $\Q_{>0}\ni h<h_{\K}$, if $Q\in \Q_p[X]$ with $Q^*(0)\in \Q_p^\times$ and the slope of $Q$ is $\leq h$, then $Q^*(U_p)$ acts on $C^{\tol}_{\K}$ invertibly. Write $Q=a_0+a_1X+\cdots + a_nX^n$. The two conditions on $Q$ means \begin{enumerate}
    \item[$\bullet$] $a_n\in \Q_p^\times$
    \item[$\bullet$] $v_p(a_n)-v_p(a_i)\leq (n-i)h$ for all $i=0, ..., n-1$.
\end{enumerate} Therefore, we have $$|a_i/a_n|<p^{(n-i)h}\text{ and }\left|\left|\frac{a_i}{a_n}U_p^{n-i}\right|\right|_{\K}<1.$$ Let $P(X)=-\frac{a_0}{a_n}X^n-\frac{a_1}{a_n}X^{n-1}-\cdots-\frac{a_{n-1}}{a_n}X$, then $\frac{1}{a_n}Q^*(X)=1-P(X)$. We can deduce that $||P(U_p)||_{\K}<1$ and so $Q^*(U_p)$ acts on $C_{\K}^{\tol}$ invertibly with inverse given explicitly by  $$Q^*(U_p)^{-1}=\frac{1}{a_n}\sum_{j\geq 0}P(U_p)^j.$$

Now fix $h<h_k$, then we have the corresponding decomposition $F_k^{\dagger}=Q^{\dagger}_hS^{\dagger}_h$ and $F_k^{\alg}=Q^{\alg}_h S^{\alg}_h$ and $$C_{k}^{\tol, \leq h}\twoheadrightarrow C_{k, \alg}^{\tol, \leq h}$$ with $C_k^{\tol, \leq h}=\ker Q^{\dagger, *}_h(U_p|_{C_k^{\tol}})$ and $C_{k, \alg}^{\tol, \leq h}=\ker Q_h^{\alg, *}(U_p|_{C_{k, \alg}^{\tol}})$. Let $C_{\K}^{\tol, \leq h}$ be the kernel of the surjection, then, by taking cohomology, we have the corresponding long exact sequence 
\[
\begin{tikzcd}
    \cdots\arrow[r] &  H^t(X_{\Iw^+}(\C), \K)^{\leq h}\arrow[r] &  H^t(X_{\Iw^+}(\C), D_{k}^{\dagger})^{\leq h}\arrow[ld, out = -20, in = 160] \\
    & H^t(X_{\Iw^+}(\C), \V_{\GSp_{2g}, k}^{\alg, \vee})^{\leq h}\arrow[r] &  H^{t+1}(X_{\Iw^+}(\C), \K)^{\leq h}\arrow[r] & \cdots
\end{tikzcd} .
\] The above claim shows that both $Q_h^{\dagger, *}(U_p)$ and $Q_h^{\alg, *}(U_p)$ act on $H^t(X_{\Iw^+}(\C), \K)^{\leq h}$ invertibly. Take any $\sigma\in H^t(X_{\Iw^+}(\C), \K)^{\leq h}$, the image of $Q_h^{\dagger}(U_p)\sigma$ in $H^t(X_{\Iw^+}(\C), D_k^{\dagger})^{\leq h}$ is zero, thus there exists $\sigma'\in H^{t-1}(X_{\Iw^+}(\C), \V_{\GSp_4, k}^{\alg, \vee})^{\leq h}$ whose image in $H^t(X_{\Iw^+}(\C), \K)^{\leq h}$ is $Q_h^{\dagger, *}(U_p)\sigma$. Since $Q_h^{\alg, *}(U_p)\sigma'=0$, thus $Q_h^{\alg, *}(U_p)Q_h^{\dagger, *}(U_p)\sigma=0$. This implies $\sigma=0$ so the desired isomorphism follows. 
\end{proof}

\begin{Remark}\label{Remark: different control theorem}
\normalfont The above control theorem is basically \cite[Theorem 6.4.1]{Ash-Stevens} with only a slight modification. There is another version of the control theorem by \cite[Proposition 4.3.10]{Urban-2011} (see also \cite[Theorem 3.2.5]{Hansen-PhD}), which gives a more explicit description of the bound $h_k$. However, the control theorem in \cite{Urban-2011} requires a modification on the Shimura varieties while this is not the case in our version.
\end{Remark}

\begin{Corollary}\label{Corollary: non-degeneracy of the pairing for cohomologies}
Let $\kappa = k\in \Z_{>0}^g$ be a dominant algebraic weight. Then the pairing $$\bls\cdot, \cdot\brs_{k}: H_{\Par, k}^{\tol, \leq h}\times H_{\Par, k}^{\tol, \leq h}\rightarrow \Q_p$$ is non-degenerate when $h<h_k$. %in the following two situations:\begin{enumerate}
%    \item the weight $\kappa$ is purely non-algebraic
%    \item the weight $\kappa=k\in \Z_{>0}^2$ is a dominant algebraic weight and $h< h_k$.
%\end{enumerate}
\end{Corollary}
\begin{proof}
This is an easy consequence of Proposition \ref{Proposition: non-degeneracy on algebriac representation} and Theorem \ref{Theorem: control theorem}.
\end{proof}

We conclude the paper by the following immediate consequence of Theorem \ref{Theorem: ramification locus}, Theorem \ref{Theorem: vanishing order and the ramification index} and Corollary \ref{Corollary: non-degeneracy of the pairing for cohomologies}. 

\begin{Corollary}\label{Corollary: main result of the paper}
Suppose $\bfx\in \calE_0^{\fl}$ is a good classical point, i.e., $\bfx$ satisfies the following conditions\begin{enumerate}
    \item[$\bullet$] $\bfx$ is a good point; 
    \item[$\bullet$] $\wt(\bfx) = k\in \Z_{>0}$ is a dominant algebraic weight; and 
    \item[$\bullet$] there is a slope datum $(\calU, h)$ such that $\bfx\in \calU$ and $h<h_k$.
\end{enumerate} Then \begin{enumerate}
    \item The adjoint $p$-adic $L$-function $L^{\adj}$ vanishes at $\bfx$ if and only if the weight map $\wt: \calE_0\rightarrow \calW$ is ramified at $\bfx$. 
    \item If $\bfx$ is furthermore a smooth point of $\calE_0^{\fl}$, let $e(\bfx)$ be as defined in Theorem \ref{Theorem: vanishing order and the ramification index}, then we have $\ord_{\bfx}L^{\adj} = e(\bfx)$. %let $e(\bfx)$ be the ramification index of $\wt$ at $\bfx$, then we have $\ord_{\bfx}L^{\adj} = e(\bfx)-1$.
\end{enumerate} 
\end{Corollary}

%% file: Reference.bib
@misc{Bellaiche-eigenbook, 
    author = {Joël Bella\"{i}che}, 
    title = {The Eigenbook: Eigenvarieties, families of Galois representations, $p$-adic $L$-functions},
    howpublished = { \url{http://people.brandeis.edu/~jbellaic/preprint/Eigenbook.pdf}}, 
    %year = {2010}
}

@article{Bellaiche-critical,
    author = {Joël Bella\"{i}che},
    title = {Critical $p$-adic $L$-functions}, 
    journal = {Inventiones mathematicae}, year = {2012},
    volume = {189},
    pages = {1-–60},
    doi = {https://doi.org/10.1007/s00222-011-0358-z},
}

@phdthesis{Kim,
    author = {Walter Kim}, 
    title = {Ramification Points on the Eigencurve and the Two Variable Symmetric Square $p$-adic $L$-Function},
    institute = {University of California, Berkeley}, 
    year = {2006}
}

@article{AIP-2015,
    %ISSN = {0003486X},
    %URL = {http://www.jstor.org/stable/24522945},
    %abstract = {Let p be an odd prime and g ≥ 2 an integer. We prove that a finite slope Siegel cuspidal eigenform of genus g can be p-adically deformed over the g-dimensional weight space. The proof of this theorem relies on the construction of a family of sheaves of locally analytic overconvergent modular forms.},
    author = {Fabrizio Andreatta and Adrian Iovita and Vincent Pilloni},
    journal = {Annals of Mathematics},
    number = {2},
    pages = {623--697},
    publisher = {Annals of Mathematics},
    title = {$p$-adic families of Siegel modular cuspforms},
    volume = {181},
    year = {2015}
}

@article{AIP-2018,
    author = {Fabrizio Andreatta and Adrian Iovita and Vincent Pilloni},
    title = {Le Halo Spectral},
    journal = {Ann. Sci. ENS},
    year = {2018},
    volume = {51},
    number = {3},
    page = {603--655}
}

@article{AIS-2015, 
    title={Overconvergent Eichler–Shimura isomorphisms}, 
    volume={14}, 
    DOI={10.1017/S1474748013000364}, 
    number={2}, 
    journal={Journal of the Institute of Mathematics of Jussieu}, 
    publisher={Cambridge University Press}, 
    author={Andreatta, Fabrizio and Iovita, Adrian and Stevens, Glenn}, 
    year={2015}, 
    pages={221–274}}

@article{Urban-2011,
    author = {Eric Urban},
    journal = {Annals of Mathematics},
    number = {3},
    pages = {1685--1784},
    publisher = {Annals of Mathematics},
    title = {Eigenvarieties for reductive groups},
    volume = {174},
    year = {2011}
}

@article{Hansen-PhD, 
    author = {David Hansen},
    title = {Universal eigenvarieties, trianguline Galois representations, and $p$-adic Langlands functoriality},
    journal = {Journal für die reine und angewandte Mathematik}, 
    year = {2017},
    doi = {https://doi.org/10.1515/crelle-2014-0130},
    issue = {730},
    pages = {1--64}
}

@misc{Hansen-notes, 
    author = {David Hansen}, 
    title = {Pairings on modules of analytic distributions}, 
    year = {2012},
    howpublished = {Unpublished notes. \url{http://www.davidrenshawhansen.com/pairing.pdf}}

}

@article{Johansson-Newton,
    title = "Extended eigenvarieties for overconvergent cohomology",
    %abstract = "Recently, Andreatta, Iovita and Pilloni have constructed spaces of overconvergent modular forms in characteristic p, together with a natural extension of the Coleman--Mazur eigencurve over a compactified (adic) weight space. Similar ideas have also been used by Liu, Wan and Xiao to study the boundary of the eigencurve. This all goes back to an idea of Coleman.In this article, we construct natural extensions of eigenvarieties for arbitrary reductive groups G over a number field which are split at all places above p. If G is GL(2)/Q, then we obtain a new construction of the extended eigencurve of Andreatta--Iovita--Pilloni. If G is an inner form of GL(2) associated to a definite quaternion algebra, our work gives a new perspective on some of the results of Liu--Wan--Xiao.We build our extended eigenvarieties following Hansen's construction using overconvergent cohomology. One key ingredient is a definition of locally analytic distribution modules which permits coefficients of characteristic p (and mixed characteristic). When G is GL(n) over a totally real or CM number field, we also construct a family of Galois representations over the reduced extended eigenvariety.",
    %keywords = "Eigenvarieties, Galois representations, P-adic automorphic forms, P-adic modular forms",
    author = "Christian Johansson and James Newton",
    year = "2019",
    month = "2",
    %day = "13",
    doi = "10.2140/ant.2019.13.93",
    %language = "English",
    volume = "13",
    pages = "93--158",
    journal = "Algebra and Number Theory",
    %issn = "1937-0652",
    publisher = "Mathematical Sciences Publishers",
    number = "1",
}

@article{Auslander-Buchsbaum, 
    author = {Maurice Auslander and David Buchsbaum}, 
    title = {On Ramification Theory in Noetherian Rings}, 
    journal = {American Journal of Mathematics}, 
    volume = {81}, 
    number = {3}, 
    page = {749--765},
    year = {1959}
}

@article{Barrera,
     author = {Barrera Salazar, Daniel},
     title = {Overconvergent cohomology of Hilbert modular varieties and $p$-adic $L$-functions},
     journal = {Annales de l'Institut Fourier},
     publisher = {Association des Annales de l'institut Fourier},
     volume = {68},
     number = {5},
     year = {2018},
     pages = {2177-2213},
     doi = {10.5802/aif.3206},
     %language = {en},
     %url={aif.centre-mersenne.org/item/AIF_2018__68_5_2177_0/}
}

@inbook{Coleman_Mazur, 
    place={Cambridge}, 
    series={London Mathematical Society Lecture Note Series}, 
    title={The Eigencurve}, 
    DOI={10.1017/CBO9780511662010.003}, 
    booktitle={Galois Representations in Arithmetic Algebraic Geometry}, 
    publisher={Cambridge University Press}, 
    author={Coleman, R. and Mazur, B.}, 
    editor={Scholl, A. J. and Taylor, R. L.Editors}, 
    year={1998}, 
    pages={1–114}, 
    collection={London Mathematical Society Lecture Note Series}
}

@misc{Ash-Stevens,
    author = {Avner Ash and Glenn Stevens},
    title = {$p$-adic deformations of arithmetic cohomology},
    howpublished = {Preprint. Available at \url{http://math.bu.edu/people/ghs/preprints/Ash-Stevens-02-08.pdf}},
    year = {2008}
}

@article{Borel-Serre,
    author = {Armand Borel and Jean-Pierre Serre},
    journal = {Commentarii mathematici Helvetici},
    pages = {436-483},
    title = {Corners and Arithmetic Groups},
    %url = {http://eudml.org/doc/139559},
    volume = {48},
    year = {1973},
}

@book{Munkers84,
  author = {Munkres, James R.},
  publisher = {Addison Wesley Publishing Company},
  title = {{Elements of Algebraic Topology}},
  year = {1984},
}

@misc{Stevens-MS,
    title = {Rigid Analytic Modular Symbols},
    author = {Glenn Stevens}, 
    year = {1994},
    howpublished = {Preprint. Available at \url{http://math.bu.edu/people/ghs/research.html}}
}

@article{Park, 
    title = {$p$-Adic family of half-integral weight modular forms via overconvergent Shintani lifting}, author = {Jeehoon Park}, 
    journal = {manuscripta mathematica}, 
    volume = {131}, 
    year = {2010},
    pages = {355–384},
    doi = {https://doi.org/10.1007/s00229-009-0323-y},
}

@incollection{Genestier-Tilouine,
    author = {Genestier, Alain and Tilouine, Jacques},
    title = {Syst\`emes de Taylor-Wiles pour $\GSp_4$},
    booktitle = {Formes automorphes (II) - Le cas du groupe $\GSp(4)$},
    editor = {Tilouine Jacques and Carayol Henri and Harris Michael and Vign\'eras Marie-France},
    series = {Ast\'erisque},
    publisher = {Soci\'et\'e math\'ematique de France},
    number = {302},
    year = {2005},
    pages = {177-290},
    %zbl = {1142.11036},
    %mrnumber = {2234862},
    %language = {fr},
    %url = {http://www.numdam.org/item/AST_2005__302__177_0}
}
